\documentclass{amsart}
\usepackage{subfiles}
\usepackage{amssymb,euscript,amsmath, mathrsfs,tensor}
\usepackage[dvips]{graphicx}
\usepackage[dvips]{color}
\usepackage[all,cmtip]{xy}
\usepackage{xcolor}
\usepackage[margin=1.2in]{geometry}
\usepackage{tikz-cd}
\usepackage{mathtools}

\newcounter{ENUM}

\def\un{\mathrm{un}}
\def\Lie{\mathrm{Lie}}

\title{On families of degenerate representations of $\GL_n(F)$}
\author{Johannes Girsch}
\address{Johannes Girsch, School of Mathematics and Statistics, University of Sheffield, Sheffield, S3 7RH, United Kingdom.}
\email{j.girsch@sheffield.ac.uk}

\author{David Helm}
\address{David Helm, Department of Mathematics, Imperial College, London, SW7 2AZ, United Kingdom.}
\email{dhelm@imperial.ac.uk }

\def\QQ{\mathbb{Q}}

\def\CP{\mathcal{P}}

\def\Spec{\operatorname{Spec}}
\def\ab{\operatorname{ab}}

\def\Hom{\mathrm{Hom}}
\def\Ext{\mathrm{Ext}}
\def\End{\mathrm{End}}

\def\Rep{\mathrm{Rep}}

\def\coker{\mathrm{coker}}

\def\L{\mathrm{L}}

\def\Ind{\mathrm{Ind}}
\def\GL{\mathrm{GL}}
\def\CO{\mathcal{O}}
\def\CK{\mathcal{K}}
\def\mm{\mathfrak{m}}
\def\cInd{\operatorname{c-Ind}}
\def\Rep{\operatorname{Rep}}

\newcommand{\margh}[1]{}

\newtheorem{thm}{Theorem}[section]
\newtheorem{prop}[thm]{Proposition}
\newtheorem{lemma}[thm]{Lemma}
\newtheorem{cor}[thm]{Corollary}
\newtheorem{conj}[thm]{Conjecture}

\theoremstyle{definition}
\newtheorem{definition}[thm]{Definition}

\newtheorem{rem}[thm]{Remark}

\numberwithin{equation}{section}

\begin{document}
\begin{abstract} 
We consider the stratification of the category of smooth representations of $\GL_n(F)$ (for $F$ a $p$-adic field) induced by degenerate Whittaker models.  We show that, remarkably, over algebraically closed fields of characteristic zero, the successive quotients in this stratification turn out to be module categories over {\em commutative} rings.  In fact, they are infinite products of rings of functions on smooth varieties over the ground field.  We further obtain explicit descriptions of these in terms of the Zelevinsky classification; these descriptions closely resemble the rings appearing in the Bernstein-Deligne description of the Bernstein center.
\end{abstract} 

\maketitle
\section{Introduction}

Let $F$ be a finite extension of $\QQ_p$ and ${\mathbb G}$ a connected reductive group over $F$.  Let $G$ denote the group of $F$-points of ${\mathbb G}$.  If $\pi$ is an irreducible smooth representation $\pi$ of $G$, then one may associate to $\pi$ a natural invariant known as its {\em wave front set}.  This is a subset of the set of nilpotent orbits in $\Lie(G)$, consisting of the  orbits ${\mathcal O}$ that are maximal (with respect to orbit closure) among those for which the coefficients $c_{\mathcal O}$ of the local character expansion of $\pi$ is nonzero.

Another way to associate a set of nilpotent orbits to $\pi$ is via the theory of degenerate Whittaker models.  For each nilpotent orbit ${\mathcal O}$ one has a space of (compactly supported) degenerate Whittaker functions $W_{\mathcal O}$ (we refer the reader to section~\ref{sec:whittaker} for the construction of this space in the case ${\mathbb G} = \GL_n$), and one can ask for which ${\mathcal O}$ there is a nonzero map $W_{\mathcal O} \rightarrow \pi$.  

A fundamental result of Moeglin-Waldspurger~\cite{MW} for $p$ odd, and Varma~\cite{varma} for $p=2$, relates these two questions.  More precisely, this result states that the wave front set of $\pi$ coincides with the set of maximal orbits ${\mathcal O}$ such that $\Hom_G(W_{\mathcal O},\pi)$ is nonzero, and moreover that for such orbits the dimension of this space is equal to the coefficient $c_{\mathcal O}$.

When ${\mathbb G} = \GL_n$ these results can be sharpened significantly.  In this setting nilpotent orbits are indexed by partitions, and Moeglin-Waldspurger show that the wave front set of any irreducible representation $\pi$ consists of a single orbit.  The partition corresponding to this orbit is thus a natural invariant of $\pi$.  Moreover, in this setting the relevant Hom-spaces are all one-dimensional; that is, when $\lambda(\pi)$ is the partition associate to $\pi$, then $\Hom_{\GL_n(F)}(W_{{\mathcal O}_{\lambda(\pi)}},\pi)$ has dimension one.  

The partition $\lambda(\pi)$ has several alternative interpretations.  In terms of the local Langlands correspondence, a result of Gourevich-Sahi (\cite{GS}, Theorem B) shows that it is equal to the {\em classification partition} of $\pi$ (that is, the partition conjugate to that describing the orbit of the monodromy operator on the Langlands parameter for the Zelevinsky dual of $\pi$).  On the other hand there is also an interpretation in terms of Bernstein-Zelevinsky derivatives: if $\lambda(\pi)$ is given by the sequence $\lambda_1 \geq \lambda_2 \geq \dots \geq \lambda_r$, then for each $i$, $\lambda_i$ is the largest integer such that the iterated derivative $(((\pi^{(\lambda_1)})^{(\lambda_2)})^{\dots})^{(\lambda_i)}$ is nonzero.  We thus refer to $\lambda$ as the highest derivative partition of $\pi$; these are discussed in detail in section~\ref{sec:highest}, below.

The association $\pi \mapsto \lambda(\pi)$ naturally defines a stratification of the category $\Rep(\GL_n(F))$ of smooth representations of $\GL_n(F)$.  Indeed, for each $\lambda$ one can define the full subcategories $\Rep(\GL_n(F))^{\preceq \lambda}$ (resp. $\Rep(\GL_n(F))^{\prec \lambda}$) consisting of smooth representations $V$ for which every irreducible subquotient $\pi$ of $V$ satisfies $\lambda(\pi) \preceq \lambda$ (resp. $\lambda(\pi) \prec \lambda$). (Here $\preceq$ is the natural dominance order on partitions, which corresponds to the closure ordering on nilpotent orbits.)  From these two subcategories one can also form the Serre quotient:
$$\Rep(\GL_n(F))^{\lambda} := \Rep(\GL_n(F))^{\preceq \lambda}/\Rep(\GL_n(F))^{\lambda}.$$
One can think of this quotient as a summand of the ``associated graded'' category of $\Rep(\GL_n(F))$ with respect to this stratification.  

The goal of this paper is to study the structure of the categories $\Rep(\GL_n(F))^{\lambda}$.  To this end we first construct projective objects $W'_{\lambda}$ of $\Rep(\GL_n(F))$, that are closely related to (in fact direct summands of) the spaces of compact degenerate Whittaker functions $W_{{\mathcal O}_{\lambda}}$.  Like the latter, the $W'_{\lambda}$ have the property that $\Hom(W'_{\lambda},\pi)$ is zero unless $\lambda(\pi) \preceq \lambda$, and is one-dimensional if $\lambda(\pi) = \lambda$.  Unlike the $W_{{\mathcal O}_{\lambda}}$, the representations $W'_{\lambda}$ have good finiteness properties (their projections to individual Bernstein blocks are finitely generated), and this allows us to easily establish their projectivity.  (It is possible that the $W_{{\mathcal O}_{\lambda}}$ are also projective but this is not known.  If so, one could argue with the $W_{{\mathcal O}_{\lambda}}$ in place of the $W'_{\lambda}$ in what follows.)

From the $W'_{\lambda}$ we construct projective objects $\CP_{\lambda}$ of $\Rep(\GL_n(F))^{\lambda}$.  Indeed, the inclusion of this category in $\Rep(\GL_n(F))$ is exact, and thus its left adjoint $V \mapsto V^{\preceq \lambda}$ that takes projectives to projectives.  We set $\CP_{\lambda} = (W'_{\lambda})^{\preceq \lambda}$.  This is a projective object of $\Rep(\GL_n(F))^{\preceq \lambda}$, and the functor $\Hom(\CP_{\lambda}, -)$ is an equivalence between the Serre quotient $\Rep(\GL_n(F))^{\lambda}$ and the category of modules over the endomorphism ring $E_{\lambda}$ of $\CP_{\lambda}$.

In light of this, the remainder of the paper is devoted to a study of the rings $E_{\lambda}$.  Remarkably, in many settings these rings turn out to be commutative!  Our main result,  Theorem~\ref{thm:main}, covers the case where the coefficient field is algebraically closed of characteristic zero. In this context the rings in question turn out to be (infinite products of) rings of functions on smooth varieties over the ground field, and can be described in a very concrete and explicit fashion, that closely resembles the Bernstein-Deligne description of the Bernstein center.

Our constructions work well in a modular or mixed-characteristic setting as well, and we conjecture that the rings $E_{\lambda}$ remain commutative in these contexts, although this seems much harder to prove.

The proof of this result in characteristic zero occupies sections~\ref{sec:char zero ends} through~\ref{sec:proof} of the paper, together with the appendix.  The most difficult part of the argument is the commutativity; in section~\ref{sec:char zero ends} we work locally, and reduce this statement to a computation of $I^{\preceq \lambda}$, where $I$ is a certain parabolic induction.  This computation (Theorem~\ref{thm:truncate}) is the technical heart of the paper, and proceeds via Bernstein-Zelevinsky theory; since the approach to this computation is largely disjoint from the rest of the paper we postpone the proof of this theorem to an appendix (section~\ref{sec:appendix}.)

Once commutativity is established we build a theory of families of objects of $\Rep(\GL_n(F))^{\preceq \lambda}$ in section~\ref{sec:families}, in a manner that exactly follows the development of the theory of co-Whittaker families in earlier work of the second author~\cite{whittaker}.  (Indeed, the theory of co-Whittaker families manifests in our more general setting, precisely as the case where $\lambda$ is the one-element partition.)  

The theory of families in particular gives us an interpretation of tangent vectors to $\Spec E_{\lambda}$ as families over a ring of dual numbers, which we use to give a description of this tangent space in terms of self-extensions of a certain ``standard module'' that arises as the parabolic induction of a tensor product of representations associated to Zelevinsky segments.  We compute these self-extensions in sections \ref{sec:segment} and \ref{sec:exts}; this self-extension computation (Theorem~\ref{thm:ext}) generalizes an earlier result of Chan (\cite{chan}, Lemma 7.3) and may be of independent interest.  The upshot of the computation is that $E_{\lambda}$ is a product of smooth algebras over the ground field.  Once this is established the rest of the argument is reasonably straightforward, and boils down to the fact that a map of smooth varieties that is bijective on points is an isomorphism; the details are given in section~\ref{sec:proof}.

The upshot of these results is a very precise structure theory for families of representations of type $\lambda$ over algebraically closed fields of characteristic zero.  We expect this to have several applications. For instance, the ``twisted doubling method in algebraic families'' of the first author \cite{doubling} aims to build gamma factors of pairs of families of representations of the group $G \times \GL_k$, for $G$ classical.  A key step in this construction is to replace the family of representations of $\GL_k$ with a family of representations of $\GL_{nk}$ for a suitable $n$; the representations in this latter family have highest derivative partition $(k,\dots,k)$.  We expect our structure theory for such representations to provide such a construction, and will address this question in a future revision of this paper.

Our stratification also seems to be closely related to a construction of Schneider-Zink~\cite{schneider-zink}. They also construct a stratification of $\Rep(\GL_n(F))$, indexed by partitions, that is closely related to the Langlands classification.  They show that the Serre quotients arising from this stratification corresponding to module categories over rings that are commutative {\em modulo nilpotents}.  We expect their stratification to be related to ours via a ``Zelevinsky involution''.  More precisely, we expect that Bernstein duality should provide a derived anti-equivalence on $\Rep(\GL_n(F))$ that intertwines their stratification with ours.  Such a result would imply that the rings arising in their theory are commutative as well.

\noindent
{\bf Acknowledgements:} The authors are grateful to Robert Kurinczuk, Shaun Stevens, and Eugen Hellmann for helpful conversations on the subject, as well as to Raphael Beuzart-Plessis for pointing out a simplification to a key construction of ours. The first author was supported by the Engineering and Physical Sciences Research Council EP/V001930/1 and EP/L015234/1, the EPSRC Centre for Doctoral Training in Geometry and Number Theory (The London School of Geometry and Number Theory), University College London and Imperial College London. The second author was partially supported by EPSRC New Horizons grant EP/V018744/1.

\section{The highest derivative partition} \label{sec:highest}
Henceforth let $G_n$ denote the group $\GL_n(F)$. We will be interested in representations of $G_n$ over an algebraically closed field $k$ of finite characteristic $\ell \neq p$, over the Witt vectors $W(k),$ which we denote by $\CO$, and over the field of fractions $\CK$ of $\CO$. We denote by $\Rep_{\CO}(G_n)$ the category of smooth $\CO[G_n]$-modules, and let $\Rep_k(G_n)$ and $\Rep_{\CK}(G_n)$ be the full subcategories of smooth $k[G_n]$ and $\CK[G_n]$-modules, respectively.
m
For $1 \leq k \leq n$ we denote by $V \mapsto V^{(k)}$ the $k$th derivative functor of Bernstein-Zelevinsky; it is an exact functor from $\Rep_{\CO}(G_n)$ to $\Rep_{\CO}(G_{n-k})$ (where we take $G_0$ to be the trivial group when $n=k$.)

Fix a Borel subgroup $B_n$ of $G_n$ with unipotent radical $U_n$, and a generic character $\psi: U_n \rightarrow \CO^{\times}$.  The representation $W_n = \cInd_{U_n}^{G_n} \psi$ is, up to isomorphism, independent of these choices, and is closely related to the highest derivative functor $V \mapsto V^{(n)}$.  Indeed, if $V$ is a finitely generated $\CO[G_n]$-module, or more generally a module such that $eV = V$ for $e$ a finite sum of primitive idempotents in the center $Z_n$ of $\Rep_{\CO}[G_n]$, then there is a natural isomorphism $\Hom_{G_n}(W_n,V) \cong V^{(n)}$.  (In what follows, we will refer to such $V$ as having ``finite Bernstein support''.)

This description extends to natural description of the $k$th derivative functor for all $k$.  Let $P_{k,n-k}$ be the standard parabolic of $G_n$ whose standard Levi subgroup $M_{k,n-k}$ is block diagonal with block sizes $k$ and $n-k$.  For $V$ with finite Bernstein support we then have an isomorphism:
$$V^{(k)} \cong \Hom_{G_k}(W_k, r_{G_n}^{P_{k,n-k}} V)$$
where $r_{G_n}^{P_{k,n-k}}$ denotes the parabolic restriction functor from $\Rep_{\CO}(G_n)$ to $\Rep_{\CO}(M_{k,n-k})$ and $G_k$ acts on this parabolic restriction by its identification with the $k$ by $k$ block in $M_{k,n-k}$.

From this description it follows that one can define iterated derivatives in terms of parabolic restriction as well: if $k_1 + \dots + k_r = n$, and $P$ is the standard parabolic with block sizes $k_1,\dots,k_r$, with Levi subgroup $M$, then for any $V$ with finite Bernstein support we have an isomorphism:
$$((V^{(k_1)})^{(k_2)\dots})^{(k_r)} \cong \Hom_M(W_M, r_{G_n}^P V),$$
where $W_M$ denotes the tensor product of the $W_{k_i}$, considered as a representation of $M$.  In particular this means that if $k'_i$ is any permutation of the $k_i$, then one has an isomorphism:
$$((V^{(k_1)})^{(k_2)\dots})^{(k_r)} \cong ((V^{(k'_1)})^{(k'_2)\dots})^{(k'_r)}.$$

If $V$ is an irreducible representation of $G_n$, then the highest derivative of $V$ is the largest $k$ such that $V^{(k)}$ is nonzero; in this case $k \geq 1$ and $V^{(k)}$ is irreducible.  We can thus define a finite sequence of integers $k_1, \dots, k_r$ such that $k_i$ is the highest derivative of $V_{i-1}$ for all $i$ (taking $V_0 = V$), and $k_1 + \dots + k_r = n$.  We will call this (provisionally) the {\em highest derivative sequence} of $V$.  From the commutativity of iterated derivatives proved in the previous paragraph, we immediately obtain:

\begin{lemma}
Let $V$ be an irreducible representation of $G_n$, and $k_1,\dots,k_r$ its highest derivative sequence.  Then $k_1 \geq \dots\geq k_r$; that is, the $k_i$ are a partition of $n$.
\end{lemma}

We thus refer to the sequence $\lambda = k_1, \dots, k_r$ as above as the {\em highest derivative partition} of $V$.  For $\lambda$ a partition we will denote by $V^{\lambda}$ the corresponding iterated derivative $((V^{(\lambda_1)})^{(\lambda_2)\dots})^{(\lambda_r)}$.

The highest derivative has the following useful ``multiplicativity'' property:

\begin{lemma}
Let $V$ and $W$ be irreducible representations of $G_n$ and $G_m$, respectively, and let $i$ and $j$ denote the highest derivatives of $V$ and $W$, respectively.  Then $i + j$ is the highest derivative of the parabolic induction $i_{P_{n,m}}^{G_{n+m}} V \otimes W$, and one has a natural isomorphism:
$$(i_{P_{n,m}}^{G_{n+m}} V \otimes W)^{(i + j)} \cong i_{P_{n-i,m-j}}^{G_{n+m-i-j}} V^{(i)} \otimes W^{(j)}.$$
\end{lemma}

From this we can inductively conclude a similar ``multiplicativity'' for the highest derivative partition:
\begin{cor} \label{cor:iterated derivative multiplicativity}
Let $V$ and $W$ be irreducible representations of $G_n$ and $G_m$, respectively, and let $\lambda$ and $\lambda'$ denote their highest derivative partitions.  Then $\lambda + \lambda'$ is the highest derivative partition of $i_{P_{n,m}}^{G_{n+m}} V \otimes W$, and there is a natural isomorphism:
$$(i_{P_{n,m}}^{G_{n+m}} V \otimes W)^{\lambda + \lambda'} \cong V^{\lambda} \otimes W^{\lambda'}.$$
\end{cor}

These iterated derivative functors are represented by natural projective objects in $\Rep_{\CO}(G_n)$ (at least with respect to the subcategory of objects with finite Bernstein support).  More precisely, for any $\lambda$, let $P$ be the standard parabolic with block sizes corresponding to the $\lambda_i$, and $M$ its standard Levi subgroup.  Then, if $P^{\circ}$ denotes the opposite parabolic to $P$, Bernstein's second adjointness gives an isomorphism:
$$\Hom_{G_n}(i_{P^{\circ}}^{G_n} W_M, V) \cong \Hom_M(W_M, r_{G_n}^P V),$$
and we have seen that the latter is necessarily isomorphic to $V^{\lambda}$ for any $V$ with finite Bernstein support.

Let $W'_{\lambda}$ denote the parabolic induction $i_{P^{\circ}}^{G_n} W_M$.  We have:

\begin{thm}
The $W'_{\lambda}$ are projective objects of $\Rep_{\CO}(G_n)$.  Moreover, for any primitive idempotent $e$ of the center $Z_n$, the representation $eW'_{\lambda}$ is finitely generated.
\end{thm}
\begin{proof}
Let $e$ be a primitive idempotent of $Z_n$, and let $e_M$ be a primitive summand of the corresponding idempotent of the center of $\Rep_{\CO}(M)$.  Then $e_M W_M$ is a finitely generated $\CO[M]$-module, and the parabolic induction $i_{P^{\circ}}^{G_n} e_M W_M$ is isomorphic to $e W'_{\lambda}$.  Thus $e W'_{\lambda}$ is finitely generated.  It is projective because we have an isomorphism of functors:
$$\Hom_{G_n}(e W'_{\lambda}, V) \cong (eV)^{\lambda}.$$ 
Since $W'_{\lambda}$ is the direct sum of the $eW'_{\lambda}$ it is projective as well.
\end{proof}

The representation $W'_{\lambda}$ has an alternative description as a compact induction from a closed subset.  Indeed, let $\psi_{\lambda}$ be the character of $U$ given by 
$$\psi_{\lambda}(U) = \sum\limits_{i \in S_{\lambda}} \psi(u_{i,i+1})$$
for some nontrivial character $\psi: F\rightarrow \CO^{\times}$, where the set $S_{\lambda}$ is given by 
$$S_{\lambda} = \{1, 2, \dots, n\} \setminus \{\lambda_r, \lambda_r + \lambda_{r-1}, \dots, \lambda_r + \dots + \lambda_1\}.$$
We then have:

\begin{lemma} \label{lemma:W' induction}
There is a natural isomorphism $W'_{\lambda} \cong \cInd_{U_n}^{G_n} \psi_{\lambda}.$
\end{lemma}
\begin{proof}
Let $P'$ be the standard parabolic subgroup of $G_n$ conjugate to $P_{\lambda}^{\circ}$; it is block upper triangular with block sizes $\lambda_r, \lambda_{r-1}, \dots, \lambda_1$.  Let $M'$ be the standard Levi subgroup of $P'$ and $W_{M'}$ the space of Whittaker functions on $M'$.  Then $W'_{\lambda}$ is naturally isomorphic to $i_{P'}^{G_n} W_{M'}$.

On the other hand, we can write $W_{M'} = \cInd_{U_n \cap M'}^{M'} \psi_{\lambda}|_{U_n \cap M'}$. Then $i_{P'}^{G_n} W_{M'}$ is, by definition, the representation:
$$\cInd_{P'}^{G_n} \operatorname{inf}_{M'}^{P'} \cInd_{U_n \cap M'}^{M'} \psi_{\lambda}|_{U_n \cap M"},$$ 
where $\operatorname{inf}$ denotes inflation from $M'$ to $P'$.

But we have a natural isomorphism 
$$\operatorname{inf}_{M'}^{P'} \cInd_{U_n \cap M'}^{M'} \psi_{\lambda}|_{U_n \cap M"} \cong
\cInd_{U_n}^{P'} \psi_{\lambda}$$
and so the claim follows.
\end{proof}

\section{The Zelevinsky classification}

In this section we recall the Zelevinsky classification of irreducible smooth representations of $G_n$.  We follow the approach of M{\'i}nguez-S{\'e}cherre~\cite{MS}, whose main results we now recall.

Let $\nu$ denote the cyclotomic character of $F^{\times}$, and let $\rho$ be an irreducible cuspidal representation of $G_n$ over $k$ or $\overline{\CK}.$  A {\em segment} $[a,b]_{\rho}$, for $b > a$ integers is a sequence of the form $[\nu^a \rho, \nu^{a+1}\rho, \dots, \nu^b \rho]$.  The segments $[a,b]_{\rho}$ and $[a',b']_{\rho'}$ are considered equivalent if $b-a = b' - a'$ and $\nu^a \rho$ is isomorphic to $\nu^{a'} \rho'$.  For a segment $\Delta = [a,b]_{\rho}$, where $\rho$ is a cuspidal representation of $G_n$, we let $n(\Delta) = n$, $\ell(\Delta) = b - a + 1,$ and $d(\Delta) = n(\Delta)\ell(\Delta).$  The {\em support} of $\Delta$ is the multiset $\{\nu^a \rho, \dots, \nu^b\rho\}$ of representations appearing in $\Delta$; note that over $k$ this can have multiplicities above $1$ if $b-a$ is at least the order of $\nu$ mod $\ell.$  The segment $\Delta$ is called {\em supercuspidal} if $\rho$ is a supercuspidal representation of $G_n.$  (This is automatic in characteristic zero, but nontrivial over $k$.)

To each segment $\Delta = [a,b]_{\rho}$ over $k$ (resp. $\overline{\CK}$) is attached an irreducible smooth representation $Z(\Delta)$ of $G_{d(\Delta)}$ over $k$ (resp. $\overline{\CK}$). If $\Delta$ is supercuspidal then the supercuspidal support of $Z(\Delta)$ is given by the support of $\Delta.$  Among the representations with this supercuspidal support, $Z(\Delta)$ is characterized by the fact that it has exactly one nonzero derivative, of degree $n(\Delta)$, and that $Z(\Delta)^{(n(\Delta))}$ is isomorphic to $Z(\Delta^{-}),$ where $\Delta^{-}$ is the segment $[a,b-1]_{\rho}.$  (If $\Delta^{-}$ is empty we let $Z(\Delta^{-})$ denote the one-dimensional vector space over $k$ (resp. $\overline{\CK}$), with the trivial action of the trivial group $G_0$.)

A {\em supercuspidal multisegment} over $k$ (resp. $\overline{\CK}$) is a multiset of supercuspidal segments over $k$ (resp. $\overline{\CK}$).  To a multisegment $\mm = \{\Delta_1, \dots, \Delta_r\}$ we associate the integer $d(\mm) = \sum_{j=1}^r d(\Delta_j)$, and the partition $\lambda(\mm)$ of $d(\mm)$, defined by setting $\lambda(\mm)_i = \sum_{j: \ell(\Delta_j) \geq i} n(\Delta_j)$.

Let $\mm = \{\Delta_1, \dots, \Delta_r\}$ be a supercuspidal multisegment, and let $I(\mm)$ denote the parabolic induction
$$i_P^{G_{d(\mm)}} Z(\Delta_1) \otimes Z(\Delta_2) \otimes \dots \otimes Z(\Delta_r).$$
It will be useful to us to understand the iterated derivatives of this parabolic induction.  To do this we first recall the standard partial order $\preceq$ on partitions: if $\lambda$ and $\lambda'$ are paritions of $n$, we say that $\lambda \preceq \lambda'$ if for all $i$,
$$\lambda_1 + \dots + \lambda_i \leq \lambda'_1 + \dots + \lambda'_i.$$
We have:

\begin{prop}
Let $\lambda$ be a partition of $d(\mm)$.  Then:
\begin{enumerate}
    \item $\Hom_{G_{d(\mm)}}(W_{\lambda}', I(\mm)) = 0$ unless $\lambda \preceq \lambda(\mm),$ and
    \item $\Hom_{G_{d(\mm)}}(W_{\lambda(\mm)}', I(\mm))$ is one-dimensional.
\end{enumerate}
\end{prop}
\begin{proof}
This is essentially~\cite{MS}, Proposition 9.19.  To translate their result into the language used in this paper, one must note that M{\'i}nguez-S{\'e}cherre work in terms of ``residually generic'' irreducible representations, rather than generic, but the two notions are equivalent for general linear groups.  Thus the set they call $\operatorname{Part}(\mm)$ is precisely the set of $\lambda$ such that $r_{G_{d(\mm)}}^{P_{\lambda}} I(\mm)$ contains a generic representation, and by second adjointness this is the same as the set of $\lambda$ such that $\Hom_{G_{(d(\mm))}}(W'_{\lambda}, I(\mm))$ is nonzero.
\end{proof}

From this one immediately sees:

\begin{cor}
\begin{enumerate}
\item The parabolic induction $I(\mm)$ admits a unique irreducible subquotient $Z(\mm)$ whose highest derivative partition is equal to $\lambda(\mm)$.
\item If $\lambda$ is the highest derivative partition of an irreducible subquotient $\pi$ of $I(\mm)$, then $\lambda \preceq \lambda(\mm)$.
\item $\Hom_{G_{d(\mm)}}(W'_{\lambda}, Z(\mm)) = 0$ unless $\lambda \preceq \lambda(\mm)$.
\item $\Hom_{G_{d(\mm)}}(W'_{\lambda(\mm)}, Z(\mm))$ is one-dimensional.
\end{enumerate}
\end{cor}

Moreover, one has:

\begin{thm}[\cite{MS}, Th{\'e}or{\`e}me 9.36]
The map $\mm \mapsto Z(\mm)$ gives a bijection between equivalence classes of supercuspidal multisegments $\mm$ with $d(\mm) = n$ over $k$ (resp. $\overline{\CK})$ and irreducible smooth representations of $G_n$ over $k$ (resp. $\overline{\CK}$).
\end{thm}

In particular, one can conclude:
\begin{cor} \label{cor:iterated derivative multiplicity}
Let $\pi$ be an irreducible representation of $G_n$, over an algebraically closed field of characteristic $\ell$ or zero, and $\lambda$ its highest derivative partition.  Then:
\begin{enumerate}
\item $\Hom_{G_{d(\mm)}}(W'_{\lambda}, \pi) = 0$ unless $\lambda \preceq \lambda(\mm)$.
\item $\Hom_{G_{d(\mm)}}(W'_{\lambda(\mm)}, \pi)$ is one-dimensional.
\end{enumerate}
\end{cor}

\section{Degenerate Whittaker models} \label{sec:whittaker}

The highest derivative partition is closely related to theory of {\em degenerate Whittaker models}, and also to the notion of the {\em wave front set}.  The relationship between these objects was first studied by Moeglin-Waldspurger~\cite{MW}.  To any pair $(S,\phi)$, where $\phi$ is a nilpotent element of $\Lie(G_n) = M_n(F)$ and $S$ is a semisimple element of $\Lie(G_n)$ such that $[S,\phi] = -2\phi$, Moeglin-Waldspurger associate a unipotent subgroup $U_{S,\phi}$ of $G_n$, and a character $\psi_{S,\phi}$ of $U_{S,\phi}$ (that depends on our fixed choice of $\psi$ above). Let $W_{S,\phi}$ be the smooth $\CO[G_n]$-module $\cInd_{U_{S,\phi}}^{G_n} \psi_{S,\phi}$ (where we regard $\psi_{S,\phi}$ as a free $\CO$-module of rank one).  We refer the reader to~\cite{GGS}, particularly section 2.5, for the details of the construction; note that they work with coefficients in $\mathbb C$ rather than $\CO$.  The notation used here is that of \cite{GGS} and not that of \cite{MW}.

Up to isomorphism, the space $W_{S,\phi}$ depends only on the conjugacy class of $\phi$, and not on the choices of $S$ or $\psi.$  Since the conjugacy classes of nilpotent orbits are in bijection with partitions of $n$, via the Jordan decomposition, for any partition $\lambda$ of $n$ we will let $W_{\lambda}$ denote a representation of the form $W_{S,\phi}$ for $\phi$ in the conjugacy class associated to $\lambda.$  We call $W_{\lambda}$ the ``space of degenerate Whittaker functions'' attached to $\lambda.$

A natural question to ask is, given an irreducible smooth representation $\pi$ of $G_n$ (over $k$ or $\overline{\CK}$), for which $\lambda$ do there exist nontrivial maps $W_{\lambda} \rightarrow \pi$?  Over fields of characteristic zero, this question has an answer in terms of the {\em classification partition} of $\pi$:

\begin{definition} Let $\pi$ be an irreducible smooth representation of $G_n$ over $\overline{\CK},$ let $\pi^t$ denote the representation obtained by applying the Zelevinsky involution to $\pi$, and let $(\phi,N)$ be the Weil-Deligne representation associated to $\pi^t$ via the local Langlands correspondence.  (In particular, $N$ is a nilpotent element of $\GL_n(\overline{\CK})$.) The {\em classification partition} of $\pi$ is the partition that is conjugate to the partition associated to the conjugacy class of $N$.
\end{definition}

In the process of the proof of~\cite{GS}, Theorem B, Gourevitch-Sahi show:

\begin{prop}
Let $\pi$ be an irreducible smooth $\overline{\CK}$-representation of $G_n$.  Then the classification partition of $\pi$ is equal to its highest derivative partition.
\end{prop}

The proof is essentially an application of the Zelevinsky classification and the additivity of the highest derivative partition with respect to parabolic induction.  We refer the reader to section 6 of~\cite{GS} for the details.

The following is Theorem B of~\cite{GS}; rephrased in terms of highest derivative paritions; it is also implicit in section II.2 of~\cite{MW}:

\begin{thm} \label{thm:char zero multiplicity}
Let $\pi$ be an irreducible smooth representation of $G_n$ over $\overline{\CK}$ and let $\lambda$ be its highest derivative partition.  Then:
\begin{enumerate}
    \item If $\lambda'$ is any partition such that $\Hom_{G_n}(W_{\lambda'}, \pi)$ is nonzero, then $\lambda' \preceq \lambda.$
    \item $\Hom_{G_n}(W_{\lambda}, \pi)$ is one-dimensional.
\end{enumerate}
\end{thm}

In fact, this result continues to hold in characteristic $\ell$.  More precisely, we have:

\begin{thm} \label{thm:modular multiplicity}
Let $\pi$ be an irreducible smooth representation of $G_n$ over $k$, with highest derivative partition $\lambda$.  Then:
\begin{enumerate}
    \item $\Hom_{G_n}(W_{\lambda},\pi)$ is one-dimensional over $k$, and
    \item if $\lambda'$ is any partition such that $\Hom_{G_n}(W_{\lambda'}, \pi)$ is nonzero, then $\lambda' \preceq \lambda.$
\end{enumerate}
\end{thm}

We will prove this theorem by lifting from characteristic $\ell$ to characteristic zero and applying Theorem~\ref{thm:char zero multiplicity}.  Let $\tilde \pi$ be an irreducible smooth representation of $G_n$ over $\overline{\CK}$, and let $\CK'$ be a finite extension of $\CK$ such that $\tilde \pi$ is defined over $\CK'$.  Let $\pi$ be a model of $\tilde \pi$ over $\CK'$; that is, a smooth $\CK'[G_n]$-module such that $\pi \otimes_{\CK'} \overline{\CK}$ is isomorphic to $\tilde \pi$.  Let $\CO'$ be the integral closure of $\CO$ in $\CK'$.  We will say that $\tilde \pi$ is {\em integral} if there exists a $G_n$-stable $\CO'$-lattice $L$ in $\pi$ (this condition is independent of the choices of $\CK'$ and the model $\pi$).  If this is the case, we define $r_{\ell} \pi := L \otimes_{\CO'} k$.  The $k[G_n]$-module $r_{\ell} \pi$ depends on the choice of lattice $L$, but its semisimplification is independent of these choices.

We first show:
\begin{lemma} \label{lemma:near projectivity}
For any exact sequence:
$$0 \rightarrow U \rightarrow V \rightarrow W \rightarrow 0$$
of admissible $\CO[G_n]$-modules, the sequence:
$$0 \rightarrow \Hom_{G_n}(W_{\lambda}, U) \rightarrow \Hom_{G_n}(W_{\lambda}, V) \rightarrow \Hom_{G_n}(W_{\lambda}, W) \rightarrow 0$$
is exact.  In particular, let
$\tilde \pi$ be an irreducible, smooth, integral representation of $G_n$ over $\overline{\CK}$.  Then $$\dim_{\overline{\CK}} \Hom_{G_n}(W_{\lambda},\pi) = \dim_k \Hom_{G_n}(W_{\lambda}, r_{\ell} \tilde \pi).$$
\end{lemma}
\begin{proof}
For any smooth $\CO[G]$-module $V$, let $V^{\vee}$ denote the smooth vectors in $\Hom_{\CO}(V,\CK/\CO);$ then $V \mapsto V^{\vee}$ is an exact contravariant functor and one has natural isomorphisms:
$$\Hom_{G_n}(V,W^{\vee}) \cong \Hom_{G_n}(W,V^{\vee})$$
for any $W,V$ smooth $\CO[G]$-modules.  If $V$ is an admissible $\CO[G]$-module, then the natural map
$V \rightarrow (V^{\vee})^{\vee}$ is an isomorphism.

Note that $W_{\lambda}^{\vee}$ is injective.  Indeed, we have $W_{\lambda} = \cInd_{U_{S,\phi}}^{G_n} \psi_{S,\phi}$ for some $S$ and $\phi$ as in the previous section, so $W_{\lambda}^{\vee}$ is isomorphic to $\Ind_{U_{S,\phi}}^{G_n} \psi_{S,\phi}^{-1}$, where $\psi_{S,\phi}^{-1}$ has underlying module $\CK'/\CO'$ and a $U_{S,\phi}$ action via the inverse of $\psi_{S,\phi}.$  In particular $\psi_{S,\phi}^{-1}$ is an injective $\CO[U_{S,\phi}]$-module; since induction from a closed subgroup is a right adjoint it follows that $W_{\lambda}^{\vee}$ is injective as a $\CO[G_n]$-module.

Applying the exact functor $\Hom_{G_n}(-, W_{\lambda}^{\vee})$ to the exact sequence:
$$0 \rightarrow W^{\vee} \rightarrow V^{\vee} \rightarrow U^{\vee} \rightarrow 0$$
we obtain an exact sequence:
$$0 \rightarrow \Hom_{G_n}(U^{\vee}, W_{\lambda}^{\vee}) \rightarrow \Hom_{G_n}(V^{\vee}, W_{\lambda}^{\vee}) \rightarrow \Hom_{G_n}(W^{\vee}, W_{\lambda}^{\vee}) \rightarrow 0.$$
Invoking the isomorphisms 
$$\Hom_{G_n}(U^{\vee}, W_{\lambda}^{\vee}) \cong \Hom_{G_n}(W_{\lambda}, (U^{\vee})^{\vee}) \cong \Hom_{G_n}(W_{\lambda},U)$$
(and the analogous isomorphisms for $V$ and $W$,) we obtained the desired result.

The final statement follows by applying this exactness to the sequence:
$$0 \rightarrow L \rightarrow L \rightarrow r_{\ell} (\tilde \pi) \rightarrow 0,$$
where $L$ is a $\CO'$ lattice in some model $\pi$ of $\tilde \pi$ over a suitable finite extension $\CK'$ of $\CK.$
\end{proof}

\begin{rem} \rm
The lemma makes it look quite plausible that the spaces $W_{\lambda}$ are actually projective, (and indeed, Gomez, Gourevitch, and Sahi have conjectured that this is the case) but this does not follow directly from the argument of the lemma and as far as we are aware the question remains open.  Indeed, even though we can prove that $W_{\lambda}^{\vee}$ is an injective $\CO'[G]$-module, it does not follow that $W_{\lambda}$ is projective; the most we can deduce from this injectivity is that $W_{\lambda}$ is flat as a $\CO[G_n]$-module.  When $\lambda$ is the maximal partition $(n)$ of $n$, then $W_{\lambda}$ is a direct sum of finitely generated $\CO[G_n]$-modules, and in this case projectivity of $W_{\lambda}$ follows from the fact that finitely generated flat smooth $\CO[G_n]$-modules are projective.  In general however $W_{\lambda}$ does not have good finiteness properties, and such an argument cannot be made to work.
\end{rem}

To obtain information about $\Hom_{G_n}(W_{\lambda},\pi)$ from lifts of $\pi$ to characteristic zero we need to understand the behaviour of $r_{\ell}$ in terms of the Zelevinsky classification.
We first recall that any supercuspidal representation $\rho$ of $G_n$ over $k$ admits an integral, supercuspidal lift ${\tilde \rho}$ over $\overline{\CK}$ such that $r_{\ell}(\tilde \rho) = \rho$.
A lift of $\Delta$ is a segment ${\tilde \Delta}$ of the form $[a,b]_{\tilde \rho}$ for some lift ${\tilde \rho}$ of $\rho$.  The key result we need is part (1) of~\cite{MS}, Th{\'e}or{\`e}me 9.39:

\begin{thm}
For any supercuspidal segment $\Delta$ over $k$, and any lift ${\tilde \Delta}$ of $\Delta$, the reduction $r_{\ell} Z(\tilde \Delta)$ is isomorphic to $Z(\Delta).$  
\end{thm}

{\em Proof of Theorem~\ref{thm:modular multiplicity}:}
We prove this by induction on the highest derivative partition $\lambda.$  If $\lambda = (1,\dots,1)$ is minimal, then it follows from the Zelevinsky classification that $\pi = Z(\Delta)$ for some supercuspidal segment $\Delta.$  We can then take ${\tilde \pi} = Z(\tilde \Delta)$; the claim then follows from Theorem~\ref{thm:char zero multiplicity} applied to ${\tilde \pi},$ followed by Lemma~\ref{lemma:near projectivity}.

Now assume the claim holds for all representations with highest derivative partition below $\lambda,$
and fix a representation $\pi = Z(\mm)$ with highest derivative partition $\lambda.$  Write $\mm = \{\Delta_1, \dots, \Delta_r\}$ and choose a lift ${\tilde \mm} = \{{\tilde \Delta}_1, \dots, {\tilde \Delta}_r\}$ of $\mm.$  Note that $I(\tilde \mm)$ and $I(\mm)$ both have highest derivative partition $\lambda.$  If $\pi'$ is a subquotient of $I(\mm)$ other than $\pi,$ then the highest derivative partition of $\pi'$ lies strictly below $\lambda$, so by our inductive hypothesis $\Hom_{G_n}(W_{\lambda'},\pi') = 0$ unless $\lambda' \prec \lambda.$

On the other hand, by Theorem~\ref{thm:char zero multiplicity}, and Lemma~\ref{lemma:near projectivity}, $\Hom_{G_n}(W_{\lambda}, I(\mm))$ is one-dimensional; from this and the previous paragraph we deduce that $\Hom_{G_n}(W_{\lambda}, \pi)$ must be one-dimensional as well.  Similarly, $\Hom_{G_n}(W_{\lambda'}, \pi)$ must be zero unless $\lambda' \prec \lambda.$ \null\hfill $\Box$

The results of this section suggest a strong connection between $W_{\lambda}$ and the spaces $W'_{\lambda}$ studied above.  In fact, this can be made completely explicit.  Indeed, one has:
\begin{prop} \label{prop:W comparison}
There is a natural surjection $W_{\lambda} \rightarrow W'_{\lambda}$, that is noncanonically split.
\end{prop}
\begin{proof}
By \cite{gourevich appendix}, Theorem A.1, there is a natural isomorphism of $W_{\lambda}$ with $\cInd_{U_{\lambda}}^{G_n} \psi_{\lambda}|_{U_{\lambda}}$, where $\psi_{\lambda}$ is the character defined in section~\ref{sec:highest} and $U_{\lambda}$ is the subgroup of $U_n$ denoted by $N_{\lambda}$ in \cite{gourevich appendix}.

We have a natural injective map:
$$\cInd_{U_n}^{G_n} \psi^{-1}_{\lambda} \hookrightarrow \Ind_{U_{\lambda}}^{G_n} \psi^{-1}_{\lambda} |_{U_{\lambda}}$$
that takes a left $U_n,\psi^{-1}_{\lambda}$-equivariant function $f$ on $G_n$ to the same function, but considered as a left $U_{\lambda}, \psi^{-1}_{\lambda}$-equivariant function.  Using the identity
$$\Hom_{G_n}(V,W^{\vee}) \cong \Hom_{G_n}(W,V^{\vee}),$$
together with the duality between $\cInd$ and $\Ind$, we see that this gives rise to a map:
$$\cInd_{U_{\lambda}}^{G_n} \psi_{\lambda}|_{U_{\lambda}} \rightarrow \Ind_{U_n}^{G_n} \psi_{\lambda}.$$

This map is given explicitly as follows: it takes a compactly supported, left $U_{\lambda},\psi_{\lambda}$-equivariant function $f$ to the function given by:
$$g \mapsto \int\limits_{U_{\lambda} \backslash U_n} \psi_{\lambda}^{-1}(u) f(ug) du.$$
In particular it takes compactly supported functions mod $U_{\lambda}$ to compactly supported functions mod $U_n$, and thus yields a map:
$$\cInd_{U_{\lambda}}^{G_n} \psi_{\lambda}|_{U_{\lambda}} \rightarrow \cInd_{U_n}^{G_n} \psi_{\lambda}.$$
This is the desired map $W_{\lambda} \rightarrow W'_{\lambda}$; it remains to show that it is surjective.  (It will then be split by projectivity of $W'_{\lambda}$.)

The dual of this map is the map:
$$\Ind_{U_n}^{G_n} \psi_{\lambda}^{-1} \rightarrow \Ind_{U_{\lambda}}^{G_n} \psi_{\lambda}^{-1}|_{U_{\lambda}}$$
that takes a function $f$ to itself, forgetting part of the equivariance.  This map is clearly injective, so the map $W_{\lambda} \rightarrow W'_{\lambda}$ is surjective as claimed.
\end{proof}

\section{The subcategory $\Rep_{\CO}(G_n)^{\preceq \lambda}$ and its localization}

For any partition $\lambda$ of $n$, we can consider the full subcategories
$\Rep_{\CO}(G_n)^{\preceq \lambda}$ and $\Rep_{\CO}(G_n)^{\prec \lambda}$, whose objects consist, respectively, of those objects $V$ for which every simple subquotient has highest derivative partition $\lambda'$ satisfying $\lambda' \preceq \lambda$ or $\lambda' \prec \lambda.$

Alternatively, one can regard $\Rep_{\CO}(G_n)^{\preceq \lambda}$ as the full subcategory of $\Rep_{\CO}(G_n)$ consisting of objects $V$ such that $\Hom_{G_n}(W'_{\lambda'}, V) = 0$ for all $\lambda'$ such that $\lambda' \not\preceq \lambda$, and $\Rep_{\CO}(G_n)^{\prec \lambda}$ as the full subcategory of $\Rep_{\CO}(G_n)^{\preceq \lambda}$ whose objects $V$ satisfy the additional condition $\Hom_{G_n}(W'_{\lambda},V) = 0$.

The inclusion of $\Rep_{\CO}(G_n)^{\preceq \lambda}$ in $\Rep_{\CO}(G_n)$ has a left adjoint $V \mapsto V^{\preceq \lambda}.$  The natural map $V \rightarrow V^{\preceq \lambda}$ is surjective, and identifies $V^{\preceq \lambda}$ with the largest quotient of $V$ that lies in $\Rep_{\CO}(G_n)^{\preceq \lambda}$.  In particular, the functor $V \mapsto V^{\preceq \lambda}$ sends finitely generated objects to finitely generated objects.  Explicitly, one has:
$$V^{\preceq \lambda} = \coker \oplus_{\lambda' \not\preceq \lambda} W'_{\lambda'} \otimes \Hom_{G_n}(W'_{\lambda'}, V) \rightarrow V.$$

As a left adjoint of an exact functor, the functor $V \mapsto V^{\preceq \lambda}$ is right exact, and takes projective objects of $\Rep_{\CO}(G_n)$ to projective objects of $\Rep_{\CO}(G_n)^{\preceq \lambda}$ (although of course such objects will not typically be projective objects of $\Rep_{\CO}(G_n).$)  In particular, the $\CO[G]$-module $(W'_{\lambda})^{\preceq \lambda}$ is a projective object of $\Rep_{\CO}(G_n)^{\preceq \lambda}.$  

Let $\Rep_{\CO}(G_n)^{\lambda}$ denote the Serre quotient of $\Rep_{\CO}(G_n)^{\preceq \lambda}$ by the Serre subcategory $\Rep_{\CO}(G_n)^{\prec \lambda}.$  This is in effect a localization of $\Rep_{\CO}(G_n)^{\preceq \lambda}$ in which we identify all objects of $\Rep_{\CO}(G_n)^{\prec \lambda}$ with the zero object, or equivalently, in which we formally invert the maps $V \rightarrow 0$ for any $V$ in $\Rep_{\CO}(G_n)^{\prec \lambda}$.  We denote by $\Hom_{G_n,\lambda}(V,W)$ the set of morphisms from $V$ to $W$ in $\Rep_{\CO}(G_n)^{\lambda}$.

\begin{lemma}
For any object $V$ of $\Rep_{\CO}(G_n)^{\lambda}$, there are a natural isomorphisms:
$$\Hom_{G_n}(W'_{\lambda}, V) \cong \Hom_{G_n}((W'_{\lambda})^{\preceq \lambda}, V) \cong
\Hom_{G_n,\lambda}((W'_{\lambda})^{\preceq \lambda}, V)$$
\end{lemma}
\begin{proof}
The first isomorphism is immediate from the definition of $W \mapsto W^{\preceq \lambda}$ as left adjoint to the inclusion functor.  For the second, note that $\Hom_{G_n,\lambda}((W'_{\lambda})^{\preceq \lambda}, V)$ is, by definition, the colimit over pairs $W$, $V'$ of $\Hom_{G_n}(W, V/V'),$ where $W$ is a subobject of $(W'_{\lambda})^{\preceq \lambda}$ such that the quotient $(W'_{\lambda})^{\preceq \lambda}/W$ lies in $\Rep_{\CO}(G_n)^{\prec \lambda}$ and $V'$ is a subobject of $V$ lying in $\Rep_{\CO}(G_n)^{\prec \lambda}$.   Since every nonzero irreducible quotient of $(W'_{\lambda})^{\preceq \lambda}$ has highest derivative partition $\lambda$, we must have $W = (W'_{\lambda})^{\preceq \lambda}.$  Moreover, since $(W'_{\lambda})^{\preceq \lambda}$ is projective, and $\Hom_{G_n}((W'_{\lambda})^{\preceq \lambda}, V') = 0$, any element of $\Hom_{G_n}((W'_{\lambda})^{\preceq \lambda}, V/V')$ lifts uniquely to a map to $V$.  Thus the colimit may be identified with $\Hom_{G_n}((W'_{\lambda})^{\preceq \lambda}, V)$ as claimed.
\end{proof}

Henceforth denote the representation $(W'_{\lambda})^{\preceq \lambda}$ by $\CP_{\lambda}$.  By Proposition~\ref{prop:W comparison} this representation is also isomorphic to $W_{\lambda}^{\preceq \lambda}$.  We let $E_{\lambda}$ denote the endomorphism ring of $\CP_{\lambda}$; the lemma implies it does not matter whether we take these endomorphisms in $\Rep_{\CO}(G_n)$ or in $\Rep_{\CO}(G_n)^{\lambda}$.

The projective $\CP_{\lambda}$ is nearly a projective generator of $\Rep_{\CO}(G_n)^{\lambda};$ it is projective and admits a nonzero map to every object of this category, but fails to be finitely generated.  To fix this, let $e$ be an idempotent of the center $Z_n$ of $\Rep_{\CO}(G_n),$ such that $e$ is a finite sum of primitive idempotents.  Then $e\CP_{\lambda}$ is a projective generator of $e\Rep_{\CO}(G_n)^{\lambda}$.  In particular the functor
$$V \mapsto \Hom_{G_n,\lambda}(e(W'_{\lambda})^{\preceq \lambda},V) \cong \Hom_{G_n}(e\CP_{\lambda}, V)$$
defines an equivalence between $\Rep_{\CO}(G_n)^{\lambda}$ and the category of $eE_{\lambda}$-modules.

It is therefore worth studying the structure of the rings $eE_{\lambda}.$  We first note:
\begin{prop}
The ring $eE_{\lambda}$ is a finitely generated $\CO$-algebra.
\end{prop}
\begin{proof}
As $e\CP_{\lambda}$ is a finitely generated $\CO[G_n]$-module, it is admissible as an $eZ_n[G_n]$-module.  In particular its endomorphisms are finitely generated as an $eZ_n$-module.  Since $eZ_n$ is a finitely generated $\CO$-algebra the claim follows.
\end{proof}

Let $eZ_{\lambda}$ be the quotient of $eZ_n$ that acts faithfully on $eE_{\lambda}$; that is, the image of the natural map $eZ_n \rightarrow E_{\lambda}.$  We conjecture the following:

\begin{conj} \label{conj:big}
\begin{enumerate}
    \item The ring $eZ_{\lambda}$ is reduced and flat over $\CO$.
    \item For each minimal prime $\eta$ of $eZ_{\lambda},$ the map $eZ_{\lambda} \rightarrow eE_{\lambda}$ becomes an isomorphism after localizing at $\eta.$
    \item The natural map:
    $$eE_{\lambda} \rightarrow \prod_{\eta} (eE_{\lambda})_{\eta}$$
    in which $\eta$ runs over the minimal primes of $eZ_{\lambda},$ is injective.  (In particular, $eE_{\lambda}$ is reduced, commutative, and flat over $\CO$.)
    \item $eE_{\lambda} \otimes_{\CO} \overline{\CK}$ is a smooth $\overline{\CK}$-algebra.
\end{enumerate}  
\end{conj}

We illustrate this conjecture when $n=2$.  When $\lambda = \{2\}$ these results are well known; in this case the natural map $eZ_{\lambda} \rightarrow eE_{\lambda}$ is an isomorphism for all $e$.  When $\lambda = \{1,1\}$, on the other hand, the full subcategory $\Rep_{\CO}(G_2)^{\preceq \lambda}$ is the full subcategory of representations on which the action of $G_2$ factors through the determinant.  This category is thus equivalent to $\Rep_{\CO}(G_1),$ and the claims of Conjecture~\ref{conj:big} are easily verified.

Over $\overline{\CK}$ we can go much further, and in fact the rings $E_{\lambda} \otimes \overline{\CK}$ can be described in a very explicit fashion.  

We first need some terminology.  We will say two segments $[a,b]_{\rho}$ and $[a',b']_{\rho'}$ are {\em inertially equivalent} if $b-a = b' - a'$ and $\rho'$ is isomorphic to an unramified twist of $\rho.$  An equivalence class of segments will be called an {\em inertial segment}.  Similarly two multisegments ${\mathfrak m}$ and ${\mathfrak m}'$ are inertially equivalent if for any inertial segment $[\Delta]$, the mutiplicities of the segments of ${\mathfrak m}$ and ${\mathfrak m'}$ in the class $[\Delta]$ are the same; we will call the inertial equivalence class of ${\mathfrak m}$ an {\em inertial multisegment,} and denote it $[{\mathfrak m}]$.  If $\pi$ is an irreducible representation of $G_n$ over $\overline{\CK}$, corresponding to a multisegment ${\mathfrak m}$, then we let $[\pi]$ denote the corresponding inertial multisegment $[\mathfrak m]$.

Note that the highest derivative partition $\lambda({\mathfrak m})$ is constant on the representatives of an inertial multisegment $[\mathfrak m]$; we denote this invariant by $\lambda([\mathfrak m])$. One similarly defines $d([\mathfrak m])$.

For an inertial multisegment $[\mathfrak m]$ we let $X_{[\mathfrak m]}$ denote the free abelian group with generators $X_{\Delta_i}$, where $\Delta_i$ are the inertial segments in $[\mathfrak m]$.  Let $W_{[\mathfrak m]}$ denote the subgroup of permutations $w$ of the segments $\Delta_i$ in $[\mathfrak m]$ such that $\Delta_i$ is inertially equivalent to $\Delta_{w(i)}$ for all $i$.  Then $W_{[\mathfrak m]}$ acts on $X_{[\mathfrak m]}$ by permuting the generators.

Normalize $[\mathfrak m]$ by demanding that $\Delta_{w(i)} = \Delta_i$ for all $i$ and all $w \in W_{[\mathfrak m]}$.  Let ${\tilde Z}(\Delta_i)$ denote the $\overline{\CK}[X_{[\mathfrak m]}][G_{d(\Delta_i)}]$-module obtained by twisting $Z(\Delta_i)$ by the character $g \mapsto X_{\Delta_i}^{v(\det g)}$ of $G_{d(\Delta_i)}$, and let ${\tilde I}([\mathfrak m])$ denote the parabolic induction:
$$i_P^{G_{d([\mathfrak m])}} {\tilde Z}(\Delta_1) \otimes {\tilde Z}(\Delta_2) \otimes \dots \otimes {\tilde Z}(\Delta_r).$$
As an object of $\Rep_{\overline{\CK}}(G_{d([\mathfrak m])})^{\preceq \lambda}$ the isomorphism class of ${\tilde I}([\mathfrak m])$ depends on the ordering of the segments; on the other hand we will show that in the localization $\Rep_{\overline{\CK}}(G_{d([\mathfrak m])})^{\lambda}$ such a parabolic induction is independent of the ordering.

To a homomorphism $f: \overline{\CK}[X_{[\mathfrak m]}] \rightarrow \overline{\CK}$ we associate the multisegment ${\mathfrak m}_f$ whose segments consists of $\Delta_1 \otimes (\chi_{f(X_1)} \circ \det), \dots, \Delta_r \otimes (\chi_{f(X_r)} \circ \det)$, where $\chi_{\alpha}$ denotes the unramified character of $F^{\times}$ that takes a uniformizer to $\alpha$.  Note that for $w \in W_{[\mathfrak m]}$ we have ${\mathfrak m}_{f \circ w} = {\mathfrak m}_f$.  This defines a bijection between the following sets:

\begin{enumerate}
    \item $W_{[\mathfrak m]}$-orbits of maps $\overline{\CK}[X_{[\mathfrak m]}] \rightarrow \overline{\CK}$, 
    \item maps $\overline{\CK}[X_{[\mathfrak m]}]^{W_{[\mathfrak m]}} \rightarrow \overline{\CK}$
    \item multisegments ${\mathfrak m}$ in the inertial equvalence class $[\mathfrak m]$, and
    \item irreducible representations $\pi$ of $G_{d([\mathfrak m])}$ over $\overline{\CK}$ with $[\pi] = [\mathfrak m]$.
\end{enumerate}
(More precisely, the bijection takes a map $f$ first to its restriction to the ring of $W_{[\mathfrak m]}$-invariants, then to the corresponding multisegment ${\mathfrak m}_f$, and finally to $Z({\mathfrak m}_f)$.)

Let $\lambda$ be a partition of $n$.  For any $[\mathfrak m]$ with $[\mathfrak m] = \lambda$,  let $\Rep_{\overline{\CK}}(G_n)^{[\mathfrak m]}$ be the full subcategory of $\Rep_{\overline{\CK}}(G_n)^{\lambda}$ consisting of representations $V$ such that every simple subobject $\pi$ of $V$ satisfies $[\pi] = [\mathfrak m]$.  We then have:

\begin{thm} \label{thm:main}
Let $\lambda$ be a partition of $n$, and let $E_{\lambda,\overline{\CK}}$ denote the endomorphism ring of $\CP_{\lambda} \otimes_{\CO} \overline{\CK}$.
\begin{enumerate}
    \item There is a natural product decomposition:
    $$\Rep_{\overline{\CK}}(G_n)^{\lambda} \cong \prod_{\lambda([\mathfrak m]) = \lambda} \Rep_{\overline{\CK}}(G_n)^{[\mathfrak m]}$$
    where the product is over equivalence classes of inertial multisegments $[\mathfrak m]$ with $\lambda(\mathfrak m) = \lambda$.
    \item The decomposition from part (1) induces a product decomposition:
    $$E_{\lambda,\overline{\CK}} \cong \prod_{\lambda([\mathfrak m]) = \lambda} E_{[\mathfrak m]}$$
    where for each $[\mathfrak m]$, we denote by $E_{[\mathfrak m]}$ the endomorphism ring of the projection of $\CP_{\lambda}$ to the direct factor $\Rep_{\overline{\CK}}(G_n)^{[\mathfrak m]}$.
    \item For each $[\mathfrak m]$, there is a natural isomorphism:
    $$E_{[\mathfrak m]} \cong \overline{\CK}[X_{[\mathfrak m]}]^{W_{[\mathfrak m]}}.$$
    (In particular the $E_{[\mathfrak m]}$ are smooth, finite type, commutative $\overline{\CK}$-algebras.)
    This isomorphism is uniquely characterized by the property that for any map:
    $$f: \overline{\CK}[X_{[\mathfrak m]}]^{W_{[\mathfrak m]}} \rightarrow \overline{\CK}$$
    one has an isomorhism:
    $$\CP_{\lambda} \otimes_{E_{\lambda},f} \overline{\CK} \cong Z({\mathfrak m}_f)$$
    in $\Rep_{\overline{\CK}}(G_n)^{[\mathfrak m]}$.
    \item Let $e_{\mathfrak m}$ be the idempotent of $E_{\lambda, \overline{\CK}}$ corresponding to the factor $E_{[\mathfrak m]}$.  There is a natural, $\overline{\CK}[X_{[\mathfrak m]}]$-equivariant isomorphism in $\Rep_{\overline{\CK}}(G_n)^{\lambda}:$
    $$e_{\mathfrak m} \CP_{\lambda} \otimes_{E_{[\mathfrak m]}} \overline{\CK}[X_{[\mathfrak m]}] \cong {\tilde I}([\mathfrak m]).$$
    In particular, ${\tilde I}([\mathfrak m])$ is a projective generator of $\Rep_{\overline{\CK}}(G_n)^{[\mathfrak m]}$ (and is independent, up to isomorphism in $\Rep_{\overline{\CK}}(G_n)^{[\mathfrak m]}$, of the ordering of the segments used to construct it).
\end{enumerate}
\end{thm}

The proof of this theorem will occupy the next several sections, concluding in section~\ref{sec:proof}.


\section{The endomorphism ring $E_{\lambda,\overline{\CK}}$} \label{sec:char zero ends}

We now turn our attention to $E_{\lambda,\overline{\CK}},$ which we will study by studying its formal completions at maximal ideals $\mathfrak p$ of the center $Z_{n,\overline{\CK}}$ of $\Rep_{\overline{\CK}}(G_n)$.  

Let $\CP_{\lambda,\overline{\CK}}$ be the projective $\CP_{\lambda} \otimes_{\CO} \overline{\CK}$ of $\Rep_{\overline{\CK}}(G_n)^{\preceq \lambda}$.  Since the projection of $\CP_{\lambda,\overline{\CK}}$ to the Bernstein block containing ${\mathfrak p}$ is an admissible $Z_{n,\overline{\CK}}[G_n]$-module, the formal completion of $E_{\lambda,\overline{\CK}}$ at ${\mathfrak p}$ may be identified with the endomorphism ring of $\widehat{(\CP_{\lambda,\overline{\CK}})}_{\mathfrak p}.$  We can regard this module as an object of the category $\Rep_{\overline{\CK}}(G_n)^{\preceq \lambda}_{\mathfrak p}$ obtained by taking the pro-completion of the full subcategory of $\Rep_{\CO}(G_n)^{\preceq \lambda}$ whose objects are the finite-length $\CO[G_n]$-modules annihilated by some power of ${\mathfrak p}.$

Standard arguments show that there are finitely many simple objects of $\Rep_{\overline{\CK}}(G_n)^{\preceq \lambda}_{\mathfrak p}$ (namely, the simple $\CO[G_n]$-modules of highest derivative partition $\lambda' \preceq \lambda$ and supercuspidal support given by ${\mathfrak p}$), and that each such simple object $\pi$ admits a projective cover $\CP_{\pi} \rightarrow \pi,$ unique up to isomorphism.  Moreover, the indecomposable projective objects are precisely the $\CP_{\pi}.$  In particular, $\widehat{(\CP_{\lambda,\overline{\CK}})}_{\mathfrak p}$ decomposes as the finite direct sum of the representations $\CP_{\pi}$, where $\pi$ is simple of highest derivative partition $\lambda$ and supercuspidal support given by ${\mathfrak p}.$

\begin{lemma} \label{lemma:independence}
Let ${\mathfrak p}$ be a maximal ideal of $Z_{n,\overline{\CK}}$, and suppose that for any two distinct simple $\CO[G_n]$-modules $\pi,\pi'$ of highest derivative $\lambda$ and supercuspidal support given by ${\mathfrak p},$ we have that $\pi'$ is not isomorphic to a subquotient of $\CP_{\pi}.$  Then we have an isomorphism:
$$\widehat{(E_{\lambda,\overline{\CK}})}_{\mathfrak p} \cong \prod_{\pi} E_{\pi}$$
where $\pi$ runs over all isomorphism classes of simple $\overline{\CK}[G_n]$-modules of highest derivative partition $\lambda$ and supercuspidal support given by ${\mathfrak p}$, and $E_{\pi}$ is the endomorphism ring of $\CP_{\pi}.$  
\end{lemma}
\begin{proof}
We have that $\widehat{(E_{\lambda,\overline{\CK}})}_{\mathfrak p}$ is the endomorphism ring of $\oplus_{\pi} \CP_{\pi}$, and the given condition implies that $\Hom_{G_n}(\CP_{\pi},\CP_{\pi'}) = 0$ unless $\pi$ and $\pi'$ are isomorphic to each other.
\end{proof}

\begin{lemma} \label{lemma:commutativity}
Let ${\mathfrak p}$ be a maximal ideal of $Z_{n,\overline{\CK}}$, let $\CP$ be a projective object of $\Rep_{\overline{\CK}}(G_n)^{\preceq \lambda}_{\mathfrak p}$ that is admissible as a $\widehat{(Z_{n,\overline{\CK}})}_{\mathfrak p}[G_n]$-module, and let $A$ be a commutative complete local Noetherian $\widehat{(Z_{n,\overline{\CK}})}_{\mathfrak p}$-algebra, with maximal ideal ${\mathfrak q}$, such that $A$ acts on $\CP$ (compatibly with the action of $\widehat{(Z_{n,\overline{\CK}})}_{\mathfrak p}$).  Suppose that:
\begin{enumerate}
    \item The cosocle of the quotient $\CP/{\mathfrak q}\CP$ is isomorphic to an absolutely irreducible representation $\pi$ of $G_n$ with highest derivative partition $\lambda$ and supercuspidal support given by $\pi.$
    \item The quotient $\CP/{\mathfrak q}\CP$ contains no subquotients isomorphic to $\pi$ other than the cosocle.
\end{enumerate}
Then there is a quotient $A'$ of $A$ such that $\CP$ is isomorphic to $\CP_{\pi} \otimes_{\widehat{(Z_{n,\overline{\CK}})}_{\mathfrak p}} A'$ as an $A[G_n]$-module, and the action of $E_{\pi}$ on $\CP_{\pi} \otimes_{\widehat{(Z_{n,\overline{\CK}})}_{\mathfrak p}} A'$ is via an injection: $E_{\pi} \rightarrow A'$ that makes $A'$ into a free, finite rank $E_{\pi}$-module.  In particular $E_{\pi}$ is commutative.

Moreover, if ${\mathfrak m}$ denotes the maximal ideal of $E_{\pi}$, there is a natural isomorphism $\CP/{\mathfrak q}\CP \cong \CP_{\pi}/{\mathfrak m}\CP_{\pi} \otimes_{\kappa({\mathfrak m})} \kappa({\mathfrak q}),$ where $\kappa({\mathfrak m})$ and $\kappa({\mathfrak q})$ denote the residue fields of ${\mathfrak m}$ and ${\mathfrak q}$, respectively.
\end{lemma}
\begin{proof}
Consider the map:
$$\CP_{\pi} \otimes \Hom_{G_n}(\CP_{\pi},\CP) \rightarrow \CP \rightarrow \CP/{\mathfrak q}\CP.$$
This map is equivariant for the action of $A$, and condition (1) implies that the composition is surjective.  It follows that $\CP$ is a finite direct sum of copies of $\CP_{\pi},$ so that
$\Hom_{G_n}(\CP_{\pi},\CP)$ is a free $E_{\pi}$-module of finite rank.  On the other hand, since $\CP_{\pi}$ is projective, we have an isomorphism:
$$\Hom_{G_n}(\CP_{\pi},\CP) \otimes_A A/{\mathfrak q} \cong \Hom_{G_n}(\CP_{\pi},\CP/{\mathfrak q}\CP)$$
of finitely generated $A[G_n]$-modules.  Moreover, the right-hand side has dimension one over $A/{\mathfrak q}$ by conditions (1) and (2), so $\Hom_{G_n}(\CP_{\pi},\CP)$ is cyclic as an $A$-module by Nakayama's lemma.  Thus if $A'$ denotes the quotient of $A$ that acts faithfully on $\Hom_{G_n}(\CP_{\pi},\CP)$, the action of $E_{\pi}$ on $\Hom_{G_n}(\CP_{\pi},\CP)$ is via a (necessarily injective) homomorphism from $E_{\pi}$ to $A'$.

We then have an isomorphism $\CP \cong \CP_{\pi} \otimes_{E_{\pi}} A'$; tensoring both sides, over $A'$, with $A'/{\mathfrak q} \cong E_{\pi}/{\mathfrak m}$, yields an isomorphism of $\CP/{\mathfrak q}\CP$ with $\CP_{\pi}/{\mathfrak m}\CP_{\pi} \otimes_{\kappa({\mathfrak m})} \kappa({\mathfrak q})$.
\end{proof}

Our strategy for studying the rings $E_{\pi}$ (and hence $E_{\lambda,\overline{\CK}}$) is to construct representations $\CP$ as in the above lemma for all irreducible representations $\pi$ of $G_n$ over $\overline{\CK}$.    Let $\pi$ be an irreducible $\overline{\CK}[G_n]$-module of highest derivative partition $\lambda$; we have $\pi = Z({\mathfrak m})$ for some multisegment ${\mathfrak m} = \{\Delta_1, \dots, \Delta_r\}$ over $\overline{\CK}.$  Write $\Delta_i = [a_i,b_i]_{\rho_i}$, and for each $1 \leq i \leq r$, and $a \leq j \leq b-1,$ let ${\tilde \rho_i \nu^j}$ denote the ``universal unramified twist'' of $\rho_i \nu^j$ over $\overline{\CK}[X_{ij}]^{\pm 1};$ that is, the representation $\rho_i \nu^j \otimes \chi_{ij}$, where $\chi_{ij}: F^{\times} \rightarrow \overline{\CK}[X_{ij}^{\pm 1}]$ is the unramified character taking a uniformizer of $F^{\times}$ to $X_{ij}.$  Then ${\tilde \rho_i \nu^j}$ is a projective $\overline{\CK}[G_{n_i}]$-module.

Let ${\tilde I}({\mathfrak m})$ denote the parabolic induction:
$$i_Q^{G_n} {\tilde \rho}_1 \nu^{b_1 -1} \chi_{1,b_1 - 1} \otimes \dots \otimes {\tilde \rho_1} \nu^{a_1} \chi_{1,a_1} \otimes \dots \otimes {\tilde \rho}_r \nu^{b_r - 1} \chi_{r,b_r - 1} \otimes \dots \otimes {\tilde \rho}_r \nu^{a_r} \chi_{r,a_r},$$
where $Q$ is a standard upper triangular parabolic subgroup of $G_n$. Moreover, we may assume that the $\Delta_i$ are orderded such that if for some $1\leq i<j\leq r$ there is an integer $\beta$ such that $\nu^{b_i+\beta}\rho_i\cong\nu^{b_j}\rho_j$ then $\beta\geq 0$. In particular this implies that $\Delta_i$ does not precede $\Delta_j$ for any $i > j$ (in the sense of~\cite{Zelevinsky}, section 4.1).  Then ${\tilde I}({\mathfrak m})$ is an admissible $\overline{\CK}[X_{ij}^{\pm 1}][G_n]$-module, and it is projective as a $\overline{\CK}[G_n]$-module by second adjointness.  Moreover, $Z_{n,\overline{\CK}}$ acts on ${\tilde I}({\mathfrak m})$ via a homomorphism $Z_{n,\overline{\CK}}\rightarrow \overline{\CK}[X_{ij}^{\pm 1}]$ that makes $\overline{\CK}[X_{ij}^{\pm 1}]$ into a finitely generated $Z_{n,\overline{\CK}}$-module.

Let ${\mathfrak q}$ be the maximal ideal of $\overline{\CK}[X_{ij}^{\pm 1}]$ generated by $X_{ij} - 1$ for all $i,j$. Then ${\mathfrak q}$ lies over the maximal ideal ${\mathfrak p}$ of $Z_{n,\overline{\CK}}$ corresponding to the supercuspidal support of $\pi$.  Set $\CP = \widehat{({\tilde I}({\mathfrak m})^{\preceq \lambda})}_{\mathfrak q}$; it is an admissible $\overline{\CK}[G_n]$-module that is a projective object of $\Rep_{\overline{\CK}}(G_n)_{\mathfrak p}^{\preceq\lambda}.$  Let $A$ be the quotient of $\widehat{\overline{\CK}[X_{ij}^{\pm 1}]}_{\mathfrak q}$ that acts faithfully on $\CP.$

In fact, the pair $(A,\CP)$ satisfies the conditions of Lemma~\ref{lemma:commutativity}.   To see this, note that for each segment $\Delta_i$, we have a natural surjection:
$$i_Q^{G_{n_i(b_i-a_i)}}\left(\nu^{b_i-1}\rho_i \otimes \dots \otimes \nu^{a_i} \rho_i \right)\rightarrow Z(\Delta_i)$$
for a suitable standard upper triangular parabolic $P$.
Combining these for all $i$ yields a surjection:
$$i_{Q'}^{G_n}\left( \nu^{b_1 - 1} \rho_1 \otimes \dots \otimes \nu^{a_1} \rho_1 \otimes \dots \otimes \nu^{b_r - 1} \rho_r \otimes \dots \otimes \nu^{a_r} \rho_r\right) \rightarrow i_{Q''}^{G_n} Z(\Delta_1) \otimes \dots \otimes Z(\Delta_r)$$
for suitable standard upper triangular parabolics $Q'$, $Q''$.  Note that the parabolic induction
$$i_{P''}^{G_n} Z(\Delta_1) \otimes \dots \otimes Z(\Delta_r)$$
is an object of $\Rep_{\overline{\CK}}(G_n)^{\preceq \lambda}_{\mathfrak p}$, and its unique irreducible subquotient of highest derivative partition $\lambda$ is isomorphic to $\pi$. Since our assumptions imply that $\Delta_i$ does not precede $\Delta_j$ for $i > j$, this subquotient appears as the cosocle of the parabolic induction; in particular this parabolic induction is indecomposable.

We then have:

\begin{thm} \label{thm:truncate}
The surjection:
$$i_{Q'}^{G_n} \nu^{b_1 - 1} \rho_1 \otimes \dots \otimes \nu^{a_1} \rho_1 \otimes \dots \otimes \nu^{b_r - 1} \rho_r \otimes \dots \otimes \nu^{a_r} \rho_r \rightarrow i_{Q''}^{G_n} Z(\Delta_1) \otimes \dots \otimes Z(\Delta_r)$$
induces an isomorphism of
$$(i_{Q'}^{G_n} \nu^{b_1 - 1} \rho_1 \otimes \dots \otimes \nu^{a_1} \rho_1 \otimes \dots \otimes \nu^{b_r - 1} \rho_r \otimes \dots \otimes \nu^{a_r} \rho_r)^{\preceq \lambda}$$
with $i_{Q''}^{G_n} Z(\Delta_1) \otimes \dots \otimes Z(\Delta_r).$
\end{thm}

The proof of this result is quite difficult and makes heavy use of the machinery of Bernstein-Zelevinsky; we postpone its proof to the appendix.  This computation has powerful implications for the structure of projectives of type $\lambda$:

\begin{lemma} \label{lemma:implications}
Let $(A,\CP)$ be the pair constructed above.  Then $(A,\CP)$ satisfies the hypotheses of Lemma~\ref{lemma:commutativity}.  Moreover, every irreducible subquotient of $\CP_{\pi}$ of highest derivative partition $\lambda$ is isomorphic to $\pi$.
\end{lemma}
\begin{proof}
Note that, since $V \mapsto V^{\preceq \lambda}$ is right exact, the quotient $\CP/{\mathfrak q}\CP$ is isomorphic to:
$$(i_{Q'}^{G_n} \nu^{b_1 - 1} \rho_1 \otimes \dots \otimes \nu^{a_1} \rho_1 \otimes \dots \otimes \nu^{b_r - 1} \rho_r \otimes \dots \otimes \nu^{a_r} \rho_r)^{\preceq \lambda}$$
and hence by Theorem~\ref{thm:truncate}, to
$i_{Q''}^{G_n} Z(\Delta_1) \otimes \dots \otimes Z(\Delta_r).$

Since the latter is indecomposable with cosocle isomorphic to $\pi$, and no other subquotients of highest derivative partition $\lambda$, the hypotheses of Lemma~\ref{lemma:commutativity} are verified.  In particular $\CP$ is a finite direct sum of copies of $\CP_{\pi},$ so it remains to show that every irreducible subquotient of $\CP$ of highest derivative partition $\lambda$ is isomorphic to $\pi.$  But every such subquotient is isomorphic to a subquotient of ${\mathfrak q}^r \CP/{\mathfrak q}^{r+1}\CP$, so this follows from the existence of a surjection:
$${\mathfrak q}^r/{\mathfrak q}^{r+1} \otimes \CP/{\mathfrak q}\CP \rightarrow {\mathfrak q}^r\CP/{\mathfrak q}^{r+1}\CP.$$
\end{proof}

\begin{cor}
Let ${\mathfrak p}$ be a maximal ideal of $Z_{n,\overline{\CK}}$ and let ${\mathfrak m}$ be a multisegment over $\overline{\CK}$ with supercuspidal support ${\mathfrak p}$ and highest derivative partition $\lambda.$  Then $\widehat{(E_{\lambda})}_{\mathfrak p}$ is commutative.
\end{cor}
\begin{proof}
By Lemma~\ref{lemma:independence} and Lemma~\ref{lemma:implications}, $\widehat{(E_{\lambda})}_{\mathfrak p}$ is isomorphic to the product of the endomorphism rings of the $\CP_{\pi}$, and these latter endomorphism rings are commutative by Lemma~\ref{lemma:commutativity} and Lemma~\ref{lemma:implications}.
\end{proof}

\section{Families of type $\lambda$} \label{sec:families}

In this section we use the projectivity of $\CP_{\lambda}$ to define certain families of representations of ``type $\lambda$'', generalizing the notion of a ``co-Whittaker family'' from~\cite{whittaker}.

\begin{definition} Let $A$ be a commutative $\CO$-algebra, and let $V$ be an admissible $A[G_n]$-module.  We say that $V$ is a {\em family of type $\lambda$} over $A$ if the following conditions are satisfied:
\begin{enumerate}
    \item $V$ lies in $\Rep_{\CO}(G_n)^{\preceq \lambda}$,
    \item the space $\Hom_{G_n}(\CP_{\lambda}, V)$ is free of rank one over $A$, and
    \item the map 
    $$\CP_{\lambda} \otimes_{\CO} \Hom_{G_n}(\CP_{\lambda}, V) \rightarrow V$$
    is surjective.
\end{enumerate}
\end{definition}

When $\lambda$ is the trivial partition $\{n\}$, then being of ``type $\lambda$'' is equivalent to being generic, and a family of type $\lambda$ in the above sense is a co-Whittaker family in the sense of~\cite{whittaker}, Definition 6.1.

Since $\CP_{\lambda}$ is projective, if $V$ is an $A[G_n]$-module and $f: A \rightarrow B$ is a morphism of Noetherian $\CO$-algebras, then the natural map:
$$\Hom_{G_n}(\CP_{\lambda}, V) \otimes_A B \rightarrow \Hom_{G_n}(\CP_{\lambda}, V \otimes_A B)$$
is an isomorphism.  In particular if $V$ is a family of type $\lambda$ over $A$, then $V \otimes_A B$ is a family of type $\lambda$ over $B$.

Consider any homomorphism $f: E_{\lambda} \rightarrow A$, where $A$ is commutative and Noetherian, and let $V = \CP_{\lambda} \otimes_{E_{\lambda},f} A$. Then $\Hom_{G_n}(\CP_{\lambda}, V)$ is free of rank one over $A$ and the map:
$$\CP_{\lambda} \otimes_{\CO} \Hom_{G_n}(\CP_{\lambda}, V) \rightarrow V$$ is surjective.  It follows that $V$ is a family of type $\lambda$ over $A$.  

This allows us to construct a ``universal'' family of type $\lambda$, in a manner exactly analogous to the construction of the universal co-Whittaker family in~\cite{whittaker}.  Indeed, let $E^{\ab}_{\lambda}$ be the quotient of $E_{\lambda}$ by its commutator ideal (or equivalently the maximal commutative quotient of $E_{\lambda}$).  Then any map from $E_{\lambda}$ to a commutative ring $A$ factors through $E^{\ab}_{\lambda}$.

The previous paragraph shows that $\CP_{\lambda} \otimes_{E_{\lambda}} E^{\ab}_{\lambda}$ is a family of type $\lambda$ over $E^{\ab}_{\lambda}.$  We will denote this family by $\CP_{\lambda}^{\ab}$.  Moreover, this family is universal in the following sense: for any family $V$ of type $\lambda$ over $A$, the action of $E_{\lambda}$ on $\Hom_{G_n}(\CP_{\lambda},V)$ induces a homomorphism $f_V: E_{\lambda} \rightarrow A$, and (after a choice of generator for $\Hom_{G_n}(\CP_{\lambda},V)$ as an $A$-module) the surjection
$$\CP_{\lambda} \otimes_{\CO} \Hom_{G_n}(\CP_{\lambda}, V) \rightarrow V$$
yields a surjection:
$$\CP_{\lambda}^{ab} \otimes_{E_{\lambda}^{\ab}} A \rightarrow V.$$  Moreover, applying the functor $\Hom_{G_n}(\CP_{\lambda}, -)$ to this surjection yields an isomorphism of free $A$-modules of rank one.

In general, if $V$ and $W$ are families of type $\lambda$ over $A$, we will say that $V$ dominates $W$ if there is a surjection from $V$ to $W$; such a surjection necessarily induces an isomorphism of $\Hom_{G_n}(\CP,V)$ with $\Hom_{G_n}(\CP,W)$.  We can make dominance into an equivalence relation by saying that $V$ is equivalent to $V'$ if there exists a $W$ that dominates both $V$ and $V'$.  Then the previous paragraph shows that the family $\CP_{\lambda}^{\ab}$ over $E_{\lambda}^{\ab}$ is universal for families of type $\lambda$ up to dominance.

Note that if we work over $\overline{\CK}$ rather than $\CO$, then $E_{\lambda} = E_{\lambda}^{\ab}$, and in such contexts we will not distinguish between the two rings.

\section{Tangent vectors} \label{sec:tangent}

We can use the theory of families developed above to describe the tangent vectors to a point of $\Spec E_{\lambda}^{\ab}.$  For a field $L$ that is a $\CO$-algebra, and a map $x: E_{\lambda}^{\ab} \rightarrow L$, we can regard $x$ as an $L$-point of $\Spec E_{\lambda}^{\ab} \otimes_{\CO} L$, and then a tangent vector to this scheme at $x$ is given by a map ${\tilde x}: E_{\lambda}^{\ab} \rightarrow L[\epsilon]/\epsilon^2$ lifting $x$.

The theory of the previous section associates to such a ${\tilde x}$ the family $\CP_{\lambda}^{\ab} \otimes_{E_{\lambda}^{\ab},{\tilde x}} L[\epsilon]/\epsilon^2.$  Denote this family by $V_{\tilde x}$, and set $V_x = V_{\tilde x}/\epsilon V_{\tilde x}$ (which can alternatively be described as $\CP_{\lambda} \otimes_{E_{\lambda}, x} L$.)  Then $V_x$ has an absolute irreducible cosocle $\pi_x$ of type $\lambda$.

We have an exact sequence:
$$0 \rightarrow \epsilon V_{\tilde x} \rightarrow V_{\tilde x} \rightarrow V_x \rightarrow 0.$$
Since multiplication by $\epsilon$ induces a surjection of $V_x$ onto $\epsilon V_{\tilde x}$, the cosocle of $V_{\tilde x}$ is isomorphic to $\pi_x$; let $W$ denote the kernel of the surjection $\epsilon V_{\tilde x} \rightarrow \pi_x$.  Then the above exact sequence induces an exact sequence:
$$0 \rightarrow \pi_x \rightarrow V_{\tilde x}/W \rightarrow V_x \rightarrow 0,$$
and hence gives a well-defined element of $\Ext^1_{G_n}(V_x,\pi_x).$

Conversely, given an element of $\Ext^1_{L[G_n]}(V_x,\pi_x)$, we have an extension:
$$0 \rightarrow \pi_x \rightarrow V \rightarrow V_x \rightarrow 0,$$
which we can make into an $L[\epsilon]/\epsilon^2$-module in which $\epsilon$ acts via the endomorphism of $V$ given by $V \rightarrow V_x \rightarrow \pi_x \hookrightarrow V$.  This makes $V$ into a family of type $\lambda$ over $L[\epsilon]/\epsilon^2$, and hence gives rise to a unique map $E_{\lambda}^{\ab} \rightarrow L[\epsilon]/\epsilon^2$ lifting $x$. Thus we obtain a bijection between tangent vectors to $\Spec E_{\lambda}^{\ab} \otimes_{\CO} L$ at $x$ and elements of $\Ext^1_{L[G_n]}(\pi_x,V_x)$.

In the particular case $L = \overline{\CK}$ (or any extension thereof) we can go further using the results of section~\ref{sec:char zero ends}.  In this setting, if the supercuspidal support of $\pi_x$ is given by a multisegment $\mathfrak m_x = \{\Delta_1,\dots,\Delta_r\}$, then we have seen that $V_x$ is given by the parabolic induction $I({\mathfrak m}_x) := i_{Q''}^{G_n} Z(\Delta_1) \otimes \dots\otimes Z(\Delta_r)$, where $Q''$ is the standard parabolic given by Theorem~\ref{thm:truncate} and the $\Delta_i$ are ordered as in the hypotheses to that theorem.

In this case we have:

\begin{prop} \label{prop:tangent}
Suppose $L$ is an extension of $\overline{\CK}$.  Then the natural map: $\Ext^1_{L[G_n]}(\pi_x,V_x) \rightarrow \Ext^1_{L[G_n]}(V_x,V_x)$ is an isomorphism.  In particular the tangent space to $\Spec E_{\lambda}$ at $x$ is isomorphic to $\Ext^1_{\L[G_n]}(V_x,V_x).$
\end{prop}
\begin{proof}
This is essentially Lemma 7.1 in~\cite{chan}; more precisely it follows by taking the smooth duals of the representations appearing in that lemma, and noting that $\Ext^1_{L[G_n]}(V,W)$ is isomorphic to $\Ext^1_{L[G_n]}(W^{\vee}, V^{\vee})$ when $V$ and $W$ are admissible.
\end{proof}

This means in particular that we can regard tangent vectors to $E_{\lambda}$ at $x$ as first-order deformations of a specific family of type $\lambda$ in the dominance class of $\pi_x$; namely the ``maximal'' such family $V_x$ over $L$.

\section{The single-segment case} \label{sec:segment}

We will use results of the previous section to show that, with coefficients in $\overline{\CK}$, the spectra of the rings $E_{\lambda}$ are smooth $\overline{\CK}$-algebras.  We begin by focusing on a special case.

Let $\chi^{\un}: G_n \rightarrow \overline{\CK}[X^{\pm 1}]^{\times}$ be the ``universal unramified character''; that is, the character taking $g$ to $X^{v(\det g)}$, where $v$ is the valuation on $F$.

As a representation of $G_n$, the character $\chi^{\un}$ has highest derivative partition $(1,1,\dots,1)$, and in fact we have:

\begin{prop} \label{prop:universal unramified}
Let $\lambda$ be the partition $(1,\dots,1)$ of $n$.  The character $\chi^{un}$ is a projective object of $\Rep_{\overline{\CK}}(G_n)^{\preceq \lambda}$.  Moreover, if $\Rep_{\overline{\CK}}(G_n)_1$ is the  block of $\Rep_{\overline{\CK}}(G_n)$ containing the trivial representation, we have an equivalence of categories between $\Rep_{\overline{\CK}}(G_n)_1^{\preceq \lambda}$  and the category of $\overline{\CK}[X^{\pm1}]$-modules, given by $V \mapsto \Hom_{G_n}(\chi^{un},V)$.
\end{prop}
\begin{proof}
We first show that the $G_n$-action on an object $V$ of $\Rep_{\overline{\CK}}(G_n)_1^{\preceq \lambda}$ factors through the determinant.  When $n=1$ this is clear.  In general, the fact that the first derivative of $V$ is the highest implies that the restriction of $V$ to the mirabolic subgroup $P_n$ is obtained by inflating a representation of $G_{n-1}$ with highest derivative partition $(1,\dots,1)$ from $G_{n-1}$ to $P_n$.  Thus by induction the restriction of $V$ to $P_n$ factors through the determinant map, and therefore so does its restriction to any conjugate of $P_n$.  Since these conjugates together generate $G_n$ the claim follows.

Conversely, any $V$ that factors through the determinant is an object of $\Rep_{\overline{\CK}}(G_n)^{\preceq \lambda}$.  Those $V$ that additionally lie in $\Rep_{\overline{\CK}}(G_n)_1$ are thus precisely those representations of $G_n$ that factor through the subgroup $G^0_n$ generated by all compact open subgroups of $G$.  Thus $\Rep_{\overline{\CK}}(G_n)_1^{\preceq \lambda}$ is naturally identified with the category of $\overline{\CK}[G_n/G_n^0]$-modules; in particular it is clear that $\chi^{\un}$ is a projective generator of this category and the result follows.
\end{proof}

This result generalizes to a class of ``minimal'' cases: fix a segment $\Delta = [a,b]_{\rho}$; its highest derivative partition is of the form $(k,\dots,k)$ for some integer $k$.  Denote this partition by $\lambda'$.

Consider the representation $Z(\Delta) \otimes \chi^{\un}$ of $G_n$; it is a smooth $\overline{\CK}[X^{\pm 1}]$-module.  Let $H$ denote the group of unramified characters $\phi$ of $G_n$ such that $\rho \otimes \phi$ is isomorphic to $\rho$.  Then (using $\chi^{\un}$ to identify $\Spec \overline{\CK}[X^{\pm 1}]$ with the torus of unramified characters of $G_n$) we have an action of $H$ on $\overline{\CK}[X^{\pm 1}]$. Fixing isomorphisms of $Z(\Delta) \otimes \phi$ with $Z(\Delta)$ for all $\phi$ in $H$, we find that $Z(\Delta) \otimes \chi^{\un}$ descends uniquely to a smooth $\overline{\CK}[X^{\pm 1}]^H$-module, which we denote by $Z(\Delta)^{\un}$, and call the ``universal unramified twist of $Z(\Delta)$''. 

Now let $\Rep_{\overline{\CK}}(G_n)_{\Delta}$ denote the block of $\Rep_{\overline{\CK}}(G_n)$ containing $Z(\Delta)$.  This block has a projective generator, obtained as the parabolic induction of the universal unramified twist of $\rho^{\otimes b-a+1}$ from a suitable Levi subgroup to $G_n$, and the endomorphisms of this parabolic induction may be naturally identified with an Iwahori Hecke algebra.  This identification yields an equivalence between $\Rep_{\overline{\CK}}(G_n)_{\Delta}$ and $\Rep_{\overline{\CK}}(\GL_{b - a + 1}(F'))_1$ for a suitable unramified extension $F'$ of $F$.
Moreover, these equivalences are naturally compatible with parabolic induction and restriction, as well as the ``top derivative'' functors.  From this it is not hard to see that these equivalences are compatible with the filtration we have constructed, in the sense that for any partition $\lambda$ of $b - a + 1$, they induce equivalences between $\Rep_{\overline{\CK}}(\GL_{b - a + 1}(F'))_1^{\preceq \lambda}$ and $\Rep_{\overline{\CK}}(G_n)_{\Delta}^{\preceq k\lambda}$.

As a result, we deduce:

\begin{prop} \label{prop:single segment}
The representation $Z(\Delta)^{\un}$ is a projective object of $\Rep_{\overline{\CK}}(G_n)^{\preceq \lambda'}$, and the functor $V \mapsto \Hom_{G_n}(Z(\Delta)^{\un}, V)$ defines an equivalence of categories between $\Rep_{\overline{\CK}}(G_n)_{\Delta}^{\preceq \lambda'}$ and the category of $\overline{\CK}[T^{\pm 1}]$-modules.
\end{prop}

We will use these results to study homological properties of arbitrary $G_n$-representations; this will entail reducing to these ``minimal'' situations via functors of the form $V \mapsto V^{\preceq \lambda'}$.  In order for this to be compatible with homological considerations we will need to understand the left derived functor $V \mapsto L^{\preceq \lambda'} V$ of this functor.  The key result we will need is:

\begin{prop} \label{prop:derived truncation}
We have $H^i(L^{\preceq \lambda'} Z(\Delta))$ = $Z(\Delta)$ if $i=0$ and zero otherwise.
\end{prop}
\begin{proof}
    
For convenience and to simplify notation, we reduce to the case of the principal block; this amounts to noting that the equivalence between $\Rep_{\overline{\CK}}(G_n)_{\Delta}$ and $\Rep_{\overline{\CK}}(\GL_{b-a+1}(F))_1$ takes $Z(\Delta)$ to the trivial character is compatible with the functors $V \mapsto V^{\preceq\lambda}$.  We thus henceforth assume that $Z(\Delta)$ is the trivial character of $G_{b-a+1}$, so that $\lambda = (1,\dots,1)$.

The trivial representation of $G_{b-a+1}$ has a projective resolution:
$$0 \rightarrow P_{b-a+1} \rightarrow \dots \rightarrow P_0 \rightarrow \overline{\CK} \rightarrow 0$$
described by Bernstein in section 4.2 of Bernstein's lecture notes \cite{rumelhart}.  The projective $P_i$ is the sum, over the dimension $i$ faces $F$ of the fundamental alcove of the Bruhat-Tits building of $G_{b-a+1}$, of $\cInd_{P_F}^{G_{b-a+1}} 1$, where $P_F$ is the pointwise stabilizer of $F$.  It is clear that $(\cInd_{P_F}^{G_{b-a+1}} 1)^{\preceq (1,\dots,1)} = \overline{\CK}$ for all facets $F$, and that the map:
$$(\cInd_{P_F}^{G_{b-a+1}} 1)^{\preceq (1,\dots,1)} \rightarrow (\cInd_{P_F'}^{G_{b-a+1}} 1)^{\preceq (1,\dots,1)}$$
induced by an inclusion of facets $F' \subset F$ is an isomorphism.  From this the result follows.
\end{proof}

\section{Self-Ext algebras and dimensions of tangent spaces} \label{sec:exts}

We now use the results of the previous section, together with the geometric lemma, to compute the self-Ext algebras $\Ext^{\bullet}_{\overline{\CK}[G_n]}(V_x,V_x)$ for $x$ a $\overline{\CK}$-point of $\Spec E_{\lambda}$.  Let ${\mathfrak m} = \{\Delta_1,\dots,\Delta_r\}$ be the corresponding multisegment, with the $\Delta_i$ ordered so that 
$$V_x = i_P^{G_n} Z(\Delta_1) \otimes \dots \otimes Z(\Delta_r).$$  

Our main result is as follows:
\begin{thm} \label{thm:ext}
The $\Ext$ algebra $\Ext^{\bullet}_{\overline{\CK}[G_n]}(V_x,V_x)$ is an exterior algebra over $\overline{\CK}$ in $r$ variables of cohomological degree one.  In particular, $\Ext^i_{\overline{\CK}[G_n]}(V_x,V_x)$ is a $\overline{\CK}$-vector space of dimension $\binom{r}{i}$ for $0 \leq i \leq r$, and is zero otherwise.
\end{thm}

From this and Proposition~\ref{prop:tangent} one deduces:
\begin{cor} \label{cor:tangent}
The tangent space to $\Spec E_{\lambda}$ at an $\overline{\CK}$-point $x$ corresponding to a multisegment $\{\Delta_1,\dots,\Delta_r\}$ containing $r$ segments is of dimension $r$.
\end{cor}

For a segment $\Delta$ appearing in ${\mathfrak m}$, let $n_{\Delta}$ be the multiplicity of $\Delta$ in ${\mathfrak m}$, and let ${\mathfrak m}_{\Delta}$ be the multisegment consisting of $n_{\Delta}$ copies of $\Delta$.  Then, for a suitable ordering $\Delta_1, \dots, \Delta_s$ of the segments appearing in ${\mathfrak m}$, and a suitable parabolic $Q = MU$ of $G_n$, we have:
$$V_x = i_Q^{G_n} Z({\mathfrak m}_{\Delta_1}) \otimes \dots \otimes Z({\mathfrak m}_{\Delta_s}).$$

One then has:
\begin{lemma}[\cite{chan}, Lemma 7.3]
The natural map: 
$$\Ext^i_{\overline{\CK}[M]}(\otimes_j Z({\mathfrak m}_{\Delta_j}), \otimes_j Z({\mathfrak m}_{\Delta_j})) \rightarrow \Ext^i_{\overline{\CK}[G_n]}(V_x,V_x)$$
induced by functoriality of parabolic induction is an isomorphism for all $i$.
\end{lemma}
\begin{proof}
Up to taking smooth duals of the representations involved, this is~\cite{chan}, Lemma 7.3; the proof is via a straightforward application of the geometric lemma.
\end{proof}

This allows us to reduce to the special case where ${\mathfrak m} = {\mathfrak m}_{\Delta}$ for some $\Delta$; that is, when ${\mathfrak m}$ consists of some number of copies of $\Delta$.  Indeed, the Levi $M$ is a product of general linear groups $G_{n_1}, \dots, G_{n_s}$.  Suppose that for each factor we knew that $\Ext^{\bullet}_{\overline{\CK}[G_{n_j}]} (Z({\mathfrak m}_{\Delta_j}),Z({\mathfrak m}_{\Delta_j}))$ was a graded exterior algebra in $n_{\Delta_j}$ generators of degree $1$.  The algebra $\Ext^i_{\overline{\CK}[M]}(\otimes_j Z({\mathfrak m}_{\Delta_j}), \otimes_j Z({\mathfrak m}_{\Delta_j}))$
is the graded tensor product of such algebras, and this would be a graded exterior algebra in $n_{\Delta_1} + \dots + n_{\Delta_s} = r$ generators of degree $1$ as Theorem~\ref{thm:ext} claims.

It thus suffices to consider the case where ${\mathfrak m}$ consists of $r$ copies of $\Delta$.  In this case $Z({\mathfrak m}) = i_P^{G_n} Z(\Delta)^{\otimes r}$, for some parabolic $P = LN$ of $G_n$.  We have a natural isomorphism:
$$\Ext^{\bullet}_{\overline{\CK}[G_n]}(Z({\mathfrak m}),Z({\mathfrak m})) \cong
\Ext^{\bullet}_{\overline{\CK}[L]}(r_{G_n}^P i_P^{G_n} Z(\Delta)^{\otimes r}, Z(\Delta)^{\otimes r}).$$

We first show:
\begin{lemma}
There exists an exact sequence of $\overline{\CK}[L]$-modules:
$$0 \rightarrow K \rightarrow r_{G_n}^P i_P^{G_n} Z(\Delta)^{\otimes r} \rightarrow Q \rightarrow 0,$$
where $Q$ has length $r!$, every irreducible subquotient of $Q$ is isomorphic to $Z(\Delta)^{\otimes r}$, and every irreducible subquotient $\pi$ of $K$ satisfies $\Ext^i_{\overline{\CK}[L]}(\pi, Z(\Delta)^{\otimes r}) = 0$ for all $i \geq 0$.

In particular we have an isomorphism: 
$$\Ext^{\bullet}_{\overline{\CK}[L]}(r_{G_n}^P i_P^{G_n} Z(\Delta)^{\otimes r}, Z(\Delta)^{\otimes r}) \cong \Ext^{\bullet}_{\overline{\CK}[L]}(Q,Z(\Delta)^{\otimes r}).$$
\end{lemma}
\begin{proof}
We apply the geometric lemma to $r_{G_n}^P i_P^{G_n} Z(\Delta)^{\otimes r}$.  One finds that in the resulting filtration there are $r!$ successive quotients of the form $Z(\Delta)^{\otimes r}$, indexed by those elements of the Weyl group of $G_n$ that normalize $L$.  Moreover, for each successive quotient $\pi$ in the filtration that is not of this form, one can check that $\Ext^i_{\overline{\CK}[L]}(\pi, Z(\Delta)^{\otimes r}) = 0$ by supercuspidal support considerations.

The geometric lemma filtration is indexed by double cosets $P \backslash G_n/P$, partially ordered via the closure relations.  Refine this to a total order, so that we obtain a filtration 
$$0 = J_0 \subseteq J_1 \subseteq \dots \subseteq J_s = r_{G_n}^P i_P^{G_n} Z(\Delta)^{\otimes r}.$$
Suppose that for some $i$ we have $J_{i+1}/J_i \cong Z(\Delta)^{\otimes r}$ and $J_{i+2}/J_{i+1}$ a successive quotient $\pi$ of the geometric lemma filtration not isomorphic to $Z(\Delta)^{\otimes r}$.  Since $\Ext^1(\pi,Z(\Delta)^{\otimes r}) = 0$ the exact sequence:
$$0 \rightarrow Z(\Delta)^{\otimes r} \rightarrow J_{i+2}/J_i \rightarrow \pi \rightarrow 0$$
splits.  Choose an isomorphism $J_{i+1}/J_i \cong \pi \oplus Z(\Delta)^{\otimes r}$, and let $J_{i+1}'$ be the preimage of the summand $\pi$ in $J_{i+2}$.  Then replacing $J_{i+1}$ with $J_{i+1}'$ gives a new filtration in which the $i$th and $i+1$st successive quotients have exchanged places.  Iterating this process eventually yields a filtration in which every successive quotient of $J_s/J_{s - r!}$ is isomorphic to $Z(\Delta)^{\otimes r}$ and every successive quotient of $J_{s-r!}$ is not.  We may thus take $K = J_{s-r!}$ and the result follows.
\end{proof}

Set $\lambda = \lambda(\Delta)$, and let $\Rep_{\overline{\CK}}(M)^{\preceq \lambda^r}$ denote the full subcategory of $Rep_{\overline{\CK}}(M)$ consisting of representations such that every irreducible subquotient is of the form $\pi_1 \otimes \dots \otimes \pi_r$, with each $\pi_i$ an irreducible representation of $G_{\frac{n}{r}}$ whose highest derivative partition $\lambda(\pi_i)$ satisfies $\lambda(\pi_i) \preceq \lambda$.  Similarly define a full subcategory $\Rep_{\overline{\CK}}(M)^{\prec \lambda^r}$ in which we additionally require that for every irreducible subquotient $\pi$ as above, there exists an $i$ with $\lambda(\pi_i) \prec \lambda$.
As usual let $\Rep_{\overline{\CK}}(M)^{\lambda^r}$ denote the corresponding Serre quotient.

Let $e$ be the primitive idempotent of $Z_{M,\overline{\CK}}$ corresponding to the inertial class of the supercuspidal support of $Z(\Delta)^{\otimes r}$.  Then $e\Rep_{\overline{\CK}}(M)^{\prec \lambda^r} = 0$, so we have an equality $e\Rep_{\overline{\CK}}(M)^{\preceq \lambda^r} = e\Rep_{\overline{\CK}}(M)^{\lambda}$.  In particular, the representation $e\CP_{\lambda^r}$ given by $e(\CP_{\lambda}^{\otimes r})$ is a projective generator of $e\Rep_{\overline{\CK}}(M)^{\preceq \lambda}$, and the functor $\Hom_M(e\CP_{\lambda^r}, -)$ induces an equivalence of this category with the category of $e(E_{\lambda,\overline{\CK}}^{\otimes r})$-modules.  By the results of the previous section this latter ring is isomorphic to $\overline{\CK}[X_1^{\pm 1}, \dots, X_r^{\pm 1}]$.  Denote this ring by $B_r$.

We have isomorphisms:
$$\Ext^{\bullet}_{\overline{\CK}[G]}(i_P^{G_n} Z(\Delta)^{\otimes r}, i_P^{G_n} Z(\Delta)^{\otimes r}) \cong
\Ext^{\bullet}_{\overline{\CK}[M]}(Q,Z(\Delta)^{\otimes r}) \cong
\Ext^{\bullet}_{\overline{\CK}[M],\lambda}(L^{\preceq \lambda^r} Q, Z(\Delta)^{\otimes r})$$
and applying the equivalence $\Hom_M(e\CP_{\lambda^r}, -)$ to the right hand side we find that it is isomorphic to $\Ext^{\bullet}_{B_r}(M_Q, B_r/{\mathfrak q})$, where $M_Q$ is the complex $\Hom_{\overline{\CK}[M]}(e\CP_{\lambda^r}, L^{\preceq \lambda^r Q})$ and ${\mathfrak q}$ is the maximal ideal $\langle X_1 - 1, \dots, X_r  - 1 \rangle$ of $B_r$.  It thus suffices to compute $M_Q$.

\begin{lemma}
The natural map $L^{\preceq \lambda^r} Q \rightarrow Q$ is an isomorphism, and thus identifies $M_Q$ with $\Hom_{\overline{\CK}[M]}(e\CP_{\lambda^r}, Q).$
\end{lemma}
\begin{proof}
The representation $Q$ has finite length, and every irreducible subquotient of $Q$ is isomorphic to $Z(\Delta)^{\otimes r}$, and by Proposition~\ref{prop:derived truncation} we have $H^i(L^{\preceq \lambda^r} Z(\Delta)^{\otimes r}) = 0$ for $i > 0$.  We thus deduce inductively that $H^i(L^{\preceq \lambda^r} Q) = 0$ for $i > 0$, and the claim follows.
\end{proof}

Thus we may regard $M_Q$ as a $B_r$-module, rather than a complex.  We have:

\begin{lemma}
The module $M_Q$ has length $r!$, and its support consists of the maximal ideal ${\mathfrak q}$.
Moreover, $\Hom_{B_r}(M_Q, B_r/{\mathfrak q})$ is one-dimensional; in particular $M_Q$ is a cyclic $B_r$-module.
\end{lemma}
\begin{proof}
Since the equivalence of $\Rep_{\overline{\CK}}(M)^{\lambda}$ with $B_r$-modules takes $Z(\Delta)^{\otimes r}$ to $B_r/{\mathfrak q}$, the first two claims follow from the fact that $Q$ has length $r!$ and every irreducible subquotient of $Q$ is isomorphic to $Z(\Delta)^{\otimes r}$.

Similarly, $\Hom_{B_r}(M_q, B_r/{\mathfrak q})$ is isomorphic to $\Hom_{\overline{\CK}[M]}(Q, Z(\Delta)^{\otimes r})$.  The latter is isomorphic to $\Hom_{\overline{\CK}[M]}(r_{G_n}^P i_P^{G_n} Z(\Delta)^{\otimes r}, Z(\Delta)^{\otimes r})$, and thus also to $\End_{\overline{\CK}[G_n]}(i_P^{G_n} Z(\Delta)^{\otimes r})$.  But $i_P^{G_n} Z(\Delta)^{\otimes r}$ is irreducible so this endomorphism ring is one-dimensional.
\end{proof}

Let $A_r$ be the subring $B_r^{S_r}$, where the symmetric group $S_r$ acts by permuting the $X_i$.  The preimage of ${\mathfrak q}$ in $S_r$ is a maximal ideal of $A_r$, which we denote by ${\mathfrak p}$.  We have:
\begin{lemma}
We have an isomorphism of $B_r$-modules:
$$M_Q \cong B_r/{\mathfrak p}B_r.$$
\end{lemma}
\begin{proof}
Since $M_Q$ is cyclic of length $r!$, and $B_r/{\mathfrak p}B_r$ also has length $r!$, it suffices to show that ${\mathfrak p}B_r$ annihilates $M_Q$.  We have an isomorphism:
$$M_Q \cong \Hom_{\overline{\CK}[M]}(\CP_{\lambda^r}, Q)$$
and the latter is a quotient of the space:
$$\Hom_{\overline{\CK}[M]}(\CP_{\lambda^r}, r_{G_n}^P i_P^{G_n} Z(\Delta)^{\otimes r}),$$ which in turn is isomorphic to 
$$\Hom_{\overline{\CK}[G_n]}(i_P^{G_n} \CP_{\lambda^r}, i_P^{G_n} Z(\Delta)^{\otimes r}).$$

Consider the action of $Z_{n,\overline{\CK}}$ on $\Hom_{\overline{\CK}[G_n]}(i_P^{G_n} \CP_{\lambda^r}, i_P^{G_n} Z(\Delta)^r)$.  On the one hand, the action on $i_P^{G_n} \CP_{\lambda^r}$ is via a map $Z_{n,\overline{\CK}} \rightarrow B_r$ whose image is $A_r$.  On the other hand, the right hand side is annihilated by an ideal of $Z_{n,\overline{\CK}}$ whose image in $A_r$ under the above map is equal to ${\mathfrak p}$.  Since every element of $\Hom_{\overline{\CK}[G_n]}(i_P^{G_n} \CP_{\lambda^r}, i_P^{G_n} Z(\Delta)^{\otimes r})$ is $Z_{n,\overline{\CK}}$-equivariant it follows that ${\mathfrak p}$ annihilates both sides as claimed.
\end{proof}

It is now straightforward to compute $\Ext^{\bullet}_{\overline{\CK}[G_n]}(i_P^{G_n} Z(\Delta)^{\otimes r}, i_P^{G_n} Z(\Delta)^{\otimes r})$.  Indeed, we have seen that there is an isomorphism of the latter with $\Ext^{\bullet}_{B_r}(M_q, B_r/{\mathfrak q})$, and using the previous lemma and the Hom-tensor adjunction we can rewrite this as $\Ext^{\bullet}_{A_r}(A_r/{\mathfrak p}, A_r/{\mathfrak p})$.  But since ${\mathfrak p}$ is a smooth dimension $r$ point of $\Spec A_r$, this latter algebra is a graded exterior algebra on $r$ generators of degree $1$.

\section{Proof of Theorem~\ref{thm:main}} \label{sec:proof}

We are now in a position to prove Theorem~\ref{thm:main}.  Let $\lambda$ be a partition of $n$.

Recall that if $[\mathfrak m]$ denotes an inertial multisegment with $\lambda([\mathfrak m]) = \lambda$, then we let ${\tilde I}([\mathfrak m])$ denote a representation of $G_n$ of the form:
$$i_P^{G_{d([\mathfrak m])}} {\tilde Z}(\Delta_1) \otimes {\tilde Z}(\Delta_2) \otimes \dots \otimes {\tilde Z}(\Delta_r).$$
We remind the reader that the isomorphism class of this representation depends on the choice of ordering of the $\Delta_i$, but this will be harmless for what follows.

We have natural isomorphisms:
$$\Hom_{\overline{\CK}[G_n]}(\CP_{\lambda}, {\tilde I}([\mathfrak m]) \cong
{\tilde I}([\mathfrak m])^{\lambda} \cong \overline{\CK}[X_{[\mathfrak m]}]$$
where the second isomorphism follows from the multiplicativity of the derivative.  The action of $E_{\lambda}$ on the left hand side induces an action of $E_{\lambda}$ on $\overline{\CK}[X_{[\mathfrak m]}]$, which is by a natural map:
$$f_{[\mathfrak m]}: E_{\lambda} \rightarrow \overline{\CK}[X_{[\mathfrak m]}].$$

We first show:
\begin{lemma}
The map $f_{[\mathfrak m]}$ takes a $\overline{\CK}$-point $x$ of $\Spec \overline{\CK}[X_{[\mathfrak m]}]$ to the $\overline{\CK}$-point of $\Spec E_{\lambda}$ giving the action of $E_{\lambda}$ on $\Hom_{\overline{\CK}[G_n]}(\CP_{\lambda}, Z({\mathfrak m}_x))$.

In particular, $f_{[\mathfrak m]}$ factors through $\overline{\CK}[X_{[\mathfrak m]}]^{W_{[\mathfrak m]}}$, and is independent of the choice of ordering of segments used to define ${\tilde I}([\mathfrak m]).$
\end{lemma}
\begin{proof}
The $W_{[\mathfrak m]}$-invariance and independence of order are immediate from the first claim, since $\overline{\CK}[X_{[\mathfrak m]}]$ is reduced.  For the first claim, note that the composition $x \circ f_{[\mathfrak m]}$ is the map giving the action of $E_{\lambda}$ on 
$$\Hom_{\overline{\CK}[G_n]}(\CP_{\lambda}, {\tilde I}([\mathfrak m]) \otimes_{\overline{\CK}[X_{[\mathfrak m]}], x} \overline{\CK}),$$
and this space is isomorphic to $\Hom_{\overline{\CK}[G_n]}(\CP_{\lambda}, Z({\mathfrak m}_x)).$
\end{proof}

Let $e$ be the primitive idempotent of $Z_{n,\overline{\CK}}$ corresponding to the block containing ${\tilde I}([\mathfrak m])$, and let $S(e,\lambda)$ be the set of equivalence classes of inertial multisegments $[{\mathfrak m}']$ such that $\lambda([{\mathfrak m}']) = \lambda$ and ${\tilde I}([{\mathfrak m}'])$ is contained in the block corresponding to $e$.  We can now show:

\begin{thm}
The map $f_{[\mathfrak m]}: E_{\lambda} \rightarrow \overline{\CK}[X_{[\mathfrak m]}]^{W_{[\mathfrak m]}}$ identifies $\Spec \overline{\CK}[X_{[\mathfrak m]}]^{W_{[\mathfrak m]}}$ with a connected component of $\Spec eE_{\lambda}$.  The product, over $S(e,\lambda)$, of the maps $f_{[{\mathfrak m}']}$ identifies $eE_{\lambda}$ with the product of the rings $\overline{\CK}[X_{[{\mathfrak m}']}]^{W_{[{\mathfrak m}']}}.$
\end{thm}
\begin{proof}
The first claim follows from the second, so we focus on the second claim.  Note that the maximal ideals of $eE_{\lambda}$ are in bijection with the irreducible representations of $G_n$ of type $\lambda$ in the block corresponding to $e$, and hence also with multisegments of type $\lambda$.  Every such multisegment is in the equivalence class of precisely one inertial multisegment $[{\mathfrak m}']$ in $S(e,\lambda)$, and the $\overline{\CK}$ points of $\Spec \overline{\CK}[X_{[\mathfrak m]}]^{W_{[{\mathfrak m}']}}$ are in bijection with the multisegments in this equivalence class.  Thus the product of the maps $f_{[{\mathfrak m}']}$ induces a bijection of $\overline{\CK}$-points between $\Spec eE_{\lambda}$ and the disjoint union of the smooth schemes $\Spec \overline{\CK}[X_{[{\mathfrak m}']}]^{W_{[{\mathfrak m}']}}.$

Let $x$ be a point of $\Spec eE_{\lambda}$, and let ${\mathfrak m}_x = \{\Delta_{1,x}, \dots, \Delta_{r_x,x}\}$ denote the corresponding multisegment; let $[{\mathfrak m}']$ denote the inertial class of ${\mathfrak m}_x$.  Then $x$ lies in the image of the map from $\Spec \overline{\CK}[X_{[{\mathfrak m}']}]^{W_{[{\mathfrak m}']}}$, and since this map is injective and its domain has dimension $r_x$, we find that the local dimension of $\Spec eE_{\lambda}$ at $x$ is at least $r_x$.  On the other hand by Corollary~\ref{cor:tangent} the dimension of the tangent space to $\Spec eE_{\lambda}$ at $x$ is $r_x$, so the local dimension is at most $r_x$.  Thus $x$ is a smooth point of dimension $r_x$.  Since $x$ was arbitrary, we find that $\Spec eE_{\lambda}$ is a smooth scheme.

Since a map of smooth $\overline{\CK}$-varieties that is a bijection on points is an isomorphism, the claim follows.
\end{proof}

From this result parts (1), (2), and (3) of Theorem~\ref{thm:main} are immediate.  On the other hand, for any $V$ in $e\Rep_{\overline{\CK}}(G_n)^{\lambda}$, one has an isomorphism:
$$\CP_{\lambda} \otimes \Hom_{\overline{\CK}[G_n]}(e\CP_{\lambda}, V) \cong V.$$
Applying this for $V = {\tilde I}([\mathfrak m])$ we see that we have an isomorphism:
$$e\CP_{\lambda} \otimes_{eE_{\lambda}} \overline{\CK}[X_{[\mathfrak m]}] \cong {\tilde I}([\mathfrak m]).$$
Since $\overline{\CK}[X_{[\mathfrak m]}]$ is a projective $eE_{\lambda}$-module, it follows that ${\tilde I}([\mathfrak m])$ is projective in $\Rep_{\overline{\CK}}(G_n)^{\lambda}$, and this completes the proof of (4).


\section{Appendix: Proof of Theorem~\ref{thm:truncate}} \label{sec:appendix}

In this appendix we will use the theory of Bernstein-Zelevinksy derivatives to prove Theorem~\ref{thm:truncate}. In particular, all representations will have coefficients in $\overline{\mathcal K}$. Firstly, we will recall some aspects of their theory and by using their techniques, deduce some results which will be needed later. 
\subsection{Prerequisites} For representations $\pi_i$ of $G_{n_i}$, where $1\leq i\leq r$, we define $\pi_1\times\dotsc\times\pi_r$ to be the normalized parabolic induction 
$$\pi_1\times\dotsc\times\pi_r\coloneqq i_{Q_{(n_1,\dotsc,n_r)}}^{G_{n_1+\dotsc+n_r}}(\pi_1\otimes\dotsc\otimes\pi_r),$$
where $Q_{(n_1,\dotsc,n_r)}$ is the standard parabolic subgroup of $G_{n_1+\dotsc+n_r}$ with blocks of sizes $n_1,\dotsc,n_r$.

Let $\rho$ be an irreducible cuspidal complex representation of $G_k$. Recall that for a nonempty segment $\Delta=[0,\alpha-1]_{\rho}=[\rho,\nu\rho,\dotsc,\nu^{\alpha-1}\rho]$ we define the length of $\Delta$ as $l(\Delta)=\alpha k$. Moreover, we set
$$I(\Delta)\coloneqq \nu^{\alpha-1}\rho\times\dotsc\times\nu\rho\times\rho$$ 
and
$$J(\Delta)\coloneqq \rho\times\nu\rho\times\dotsc\times\nu^{\alpha-1}\rho.$$
Then $I(\Delta)$ has an unique irreducible quotient which we denote by $Z(\Delta)$. The unique irreducible subrepresentation of $J(\Delta)$ is also $Z(\Delta)$. If $\Delta$ is the empty set we set $I(\Delta)$ and $J(\Delta)$ to be the trivial representation of the trivial group $G_0=\{1\}$.
For $1\leq\beta\leq\alpha-1$ we define the segment $\Delta^{-,\beta}$ as $[0,\alpha-\beta-1]_\rho$. If $\beta>\alpha-1$ let $\Delta^{-,\beta}$ be the empty set. For ease of notation we often write $\Delta^-$ for $\Delta^{-,1}$. Moreover, we set $^-\Delta$ to be the segment $[1,\alpha-1]_\rho$.

Let $P_k$ be the mirabolic subgroup of $G_k$. We often write $\pi|_{P_k}$ for the restriction of a representation $\pi$ of $G_k$ to the mirabolic subgroup $P_k$. When we say that a $P_k$-representation $\sigma$ is contained in a representation $\pi$ of $G_k$ we mean that there is an injective $P_k$-homomorphism $\sigma\hookrightarrow\pi|_{P_k}$.
Let $U_n$ be the unipotent radical of $P_n$, i.e.\ $$U_n=\{(u_{ij})\in P_n\mid u_{ij}=\delta_{ij}\text{ for }j<n\}.$$
We define a character $\psi$ of $U_n$ by setting that $\psi((u_{ij}))=\psi(u_{n-1,n})$.
We will need the following functors, originally defined by Bernstein and Zelevinsky. Let
\begin{align*}
    \Psi^-\colon\operatorname{Rep}_{\overline{\mathcal K}}(P_n)&\to\operatorname{Rep}_{\overline{\mathcal K}}(G_{n-1}),\\
    \pi&\mapsto\delta_{U_n}^{-1/2}\otimes \pi/\pi(U_n,1)\\
    \Phi^-\colon\operatorname{Rep}_{\overline{\mathcal K}}(P_n)&\to\operatorname{Rep}_{\overline{\mathcal K}}(P_{n-1}),\\
    \pi&\mapsto\delta_{U_n}^{-1/2}\otimes \pi/\pi(U_n,\psi)\\
    \Psi^+\colon\operatorname{Rep}_{\overline{\mathcal K}}(G_{n-1})&\to\operatorname{Rep}_{\overline{\mathcal K}}(P_n)\\
    \pi&\mapsto\delta_{U_n}^{1/2}\otimes\pi\text{ (where }U_n\text{ acts trivially),}\\
    \hat{\Phi}^+\colon\operatorname{Rep}_{\overline{\mathcal K}}(P_{n-1})&\to\operatorname{Rep}_{\overline{\mathcal K}}(P_n)\\
    \pi&\mapsto\operatorname{Ind}_{P_{n-1}U_n}^{P_n}(\delta_{U_n}^{1/2}\otimes\pi)\text{ (where }U_n\text{ acts via }\psi).
    \end{align*}
Their basic properties can be found in Section 3 of \cite{BZI}. Note that for $\pi\in\operatorname{Rep}_{\overline{\mathcal K}}(G_{n})$ its $r$-th Bernstein-Zelevinsky derivative equals
$$\pi^{(r)}=\Psi^{-}(\Phi^-)^{r-1}(\pi|_{P_n})\in\operatorname{Rep}_{\overline{\mathcal K}}(G_{n-r}).$$

The following is a special case of Proposition 4.6 in \cite{BZI}. 
\begin{lemma}\label{resolve}
Let $\Delta_1,\Delta_2$ be two juxtaposed segments where $\Delta_2$ precedes $\Delta_1$ and $l(\Delta_1\cup\Delta_2)=r$. There is a short exact sequence of smooth $G_{r}$-representations
	$$0\to Z(\Delta_1,\Delta_2)\hookrightarrow Z(\Delta_1)\times Z(\Delta_2)\twoheadrightarrow Z(\Delta_2\cup\Delta_1)\to 0.$$
\end{lemma}
\begin{lemma}\label{lemcomp}
	Let $\Delta=[0,\alpha-1]_\rho$ be a segment. Then
	$$Z([\nu^\alpha\rho],\Delta)^{(k)}=Z([\nu^\alpha\rho],\Delta^-).$$
\end{lemma}
\begin{proof}
	By Lemma \ref{resolve} we have that
	$$v^{\alpha}\rho\times Z(\Delta)=Z([\nu^\alpha\rho],\Delta)+Z(\Delta\cup[\nu^\alpha\rho])$$
	in the Grothendieck group. By taking the $(k)$-th derivative of this identity we obtain
	$$Z([\nu^\alpha\rho],\Delta)^{(k)}+Z(\Delta\cup[\nu^\alpha\rho])^{(k)}=Z(\Delta)+\nu^\alpha\rho\times Z(\Delta^-),$$
	and since $\Delta^-$ and $[\nu^\alpha\rho]$ are unlinked we have that $Z([\nu^\alpha\rho],\Delta^-)=\nu^\alpha\rho\times Z(\Delta^-)$ which proves the result.
\end{proof}

\begin{lemma}\label{redpar}
	 Let $\Delta_1,\Delta_2$ be two segments, where $\Delta_1=[0,\alpha-1]_\rho$ and $\alpha\geq 2$, and $\Delta_2=[\beta,-1]_\rho$ for some integer $\beta<-1$. Then there is an embedding
	$$Z(\Delta_1,\Delta_2)\hookrightarrow Z(\Delta_1^-,\Delta_2)\times\nu^{\alpha-1}\rho.$$
\end{lemma}
\begin{proof}
	Note that there is an embedding $Z(\Delta_1)\hookrightarrow Z(\Delta_1^-)\times\nu^{\alpha-1}\rho$ and hence by Lemma \ref{resolve} we have an embedding
	$$Z(\Delta_1,\Delta_2)\hookrightarrow Z(\Delta_1^-)\times\nu^{\alpha-1}\rho\times Z(\Delta_2)\cong Z(\Delta_1^-)\times Z(\Delta_2)\times\nu^{\alpha-1}\rho.$$
	Now $Z(\Delta_1^-)\times Z(\Delta_2)$ has two Jordan-H\"{o}lder factors, $Z(\Delta_1^-,\Delta_2)$ and $Z(\Delta_2\cup\Delta_1^-)$. 
	If $Z(\Delta_1,\Delta_2)$ would be contained in $Z(\Delta_2\cup\Delta_1^-)\times\nu^{\alpha-1}\rho$ by taking the $(2k)$-derivative we would obtain that
	$$Z(\Delta_1^-,\Delta_2^-)\hookrightarrow Z(\Delta_2\cup\Delta_1^{-,2}),$$
	which is absurd.
\end{proof}

\begin{lemma}\label{weirdcase}
	Let $\Delta=[0,\alpha-1]_\rho$ be a segment. Then the Jordan-H\"{o}lder constituents of 
	$$\nu^{\alpha+1}\rho\times Z([\nu^{\alpha}\rho],\Delta)$$
	are $Z([\nu^{\alpha+1}\rho],[\nu^{\alpha}\rho],\Delta)$ and $Z([\nu^{\alpha}\rho,\nu^{\alpha+1}\rho],\Delta)$.
	
\end{lemma}
\begin{proof}
Let $\pi=\nu^{\alpha+1}\rho\times Z([\nu^{\alpha}\rho],\Delta)$. Note that $Z([\nu^{\alpha+1}\rho],[\nu^{\alpha}\rho],\Delta)$ is the unique irreducible subrepresentation of $\nu^{\alpha+1}\rho\times\nu^{\alpha}\rho\times Z(\Delta)$ by Theorem 6.1 of \cite{Zelevinsky}. We have embeddings
$$Z([\nu^{\alpha+1}\rho],[\nu^{\alpha}\rho],\Delta)\hookrightarrow\pi\hookrightarrow\nu^{\alpha+1}\rho\times\nu^{\alpha}\rho\times Z(\Delta)$$
and hence $Z([\nu^{\alpha+1}\rho],[\nu^{\alpha}\rho],\Delta)$ is the unique irreducible subrepresentation of $\pi$. First note that 
	$$Z([\nu^{\alpha+1}\rho],[\nu^{\alpha}\rho],\Delta)\not\cong\pi.$$
	This follows since we also have an embedding 
	$$Z([\nu^{\alpha+1}\rho],[\nu^{\alpha}\rho],\Delta)\hookrightarrow Z([\nu^{\alpha+1}\rho],[\nu^{\alpha}\rho])\times Z(\Delta)$$ and hence would obtain an embedding
	$$\pi^{(2k)}\hookrightarrow (Z([\nu^{\alpha+1}\rho],[\nu^{\alpha}\rho])\times Z(\Delta))^{(2k)}.$$
	However, in the Grothendieck group we have that 
	$$\pi^{(2k)}=Z([\nu^{\alpha}\rho],\Delta^-)+\nu^{\alpha+1}\rho\times Z(\Delta^{-})$$
	and 
	$$(Z([\nu^{\alpha+1}\rho],[\nu^{\alpha}\rho])\times Z(\Delta))^{(2k)}=Z(\Delta)+\nu^{\alpha+1}\rho\times Z(\Delta^-)$$
	which yields a contradiction.
 
	Now let $\omega$ be a Jordan-H\"{o}lder constituent of $\pi$ different from $Z([\nu^{\alpha+1}\rho],[\nu^{\alpha}\rho],\Delta)$. Then $\omega$ has to have highest derivative either $k$, $2k$ or $3k$ since $\pi^{(i)}=0$ for $i>3k$. Note that
	$$\pi^{(3k)}=Z([\nu^{\alpha+1}\rho],[\nu^{\alpha}\rho],\Delta)^{(3k)}=Z(\Delta^-),$$
	which implies that $\omega$ cannot have highest derivative $3k$. If $\omega$ has highest derivative $k$ then we have to have that $\omega=Z(\Delta\cup[\nu^{\alpha}\rho,\nu^{\alpha+1}\rho])$ and $\omega^{(k)}=Z(\Delta\cup[\nu^{\alpha}\rho])$. However, by Lemma \ref{lemcomp} we have
	$$\pi^{(k)}=\nu^{\alpha+1}\rho\times Z([\nu^{\alpha}\rho],\Delta^-)+ Z([\nu^{\alpha}\rho],\Delta)$$ and hence $\omega$ has highest derivative $2k$. Since $\pi^{(2k)}=Z([\nu^{\alpha}\rho],\Delta^-)+\nu^{\alpha+1}\rho\times Z(\Delta^{-})$ this implies that $\omega=Z([\nu^{\alpha}\rho,\nu^{\alpha+1}\rho],\Delta)$.
\end{proof}

We will need the following variants of Propositions 5.1 and 5.3 of \cite{Zelevinsky}. We refer to Section 3 of \cite{Zelevinsky} for the necessary definitions. 
\begin{lemma}
	Let $\tau$ be a smooth representation of $P_n$. The following are equivalent:
	\begin{enumerate}
		\item All  nonzero subrepresentations of $\tau$ have highest derivative at least $m$.
		\item There is an injective $P_n$-morphism $\tau\to\bigoplus_{j\geq m}(\hat{\Phi}^+)^{j-1}\Psi^+(\tau^{(j)})$.
	\end{enumerate}
\end{lemma}
\begin{proof}
	Note that for any smooth representations $\sigma$ and $\omega$ of $P_n$ we have by Section 3.2 of \cite{BZI} that 
	\begin{equation}\label{eqfrobfunc}
		\Hom_{P_n}(\sigma,(\hat{\Phi}^+)^{r-1}\circ\Psi^+(\omega^{(r)}))\cong \Hom_{G_{n-r}}(\sigma^{(r)},\omega^{(r)})
	\end{equation}
	for $1\leq r\leq n$.
 
	(1) implies (2): For any $1\leq r\leq n$ there is a canonical map $\tau\to(\hat{\Phi}^+)^{r-1}\Psi^+(\tau^{(r)})$ which corresponds to the identity morphism $\tau^{(r)}\to\tau^{(r)}$ under the above isomorphism. Let $\Omega$ be the kernel of the induced map $\tau\to\bigoplus_{j\geq m}(\hat{\Phi}^+)^{j-1}\Psi^+(\tau^{(j)})$. For all $j\geq m$ the composition of $\Omega\hookrightarrow\tau$ and $\tau\to(\hat{\Phi}^+)^{j-1}\Psi^+(\tau^{(j)})$ is zero. Equation (\ref{eqfrobfunc}) implies that the embedding $\Omega^{(j)}\hookrightarrow\tau^{(j)}$ is the zero map and hence $\Omega^{(j)}=0$ for all $j\geq m$. By our assumption we obtain that $\Omega=0$.
 
	(2) implies (1): Let $\sigma$ be a nonzero subrepresentation of $\tau$. Then there is $r\geq m$ such that the image of $\sigma$ under the map $\tau\to (\hat{\Phi}^+)^{r-1}\Psi^+(\tau^{(r)})$ is nonzero. In particular we obtain that $\Hom_{P_n}(\sigma,(\hat{\Phi})^{r-1}\circ\Psi^+(\tau^{(r)}))$ is nonzero, which by Equation (\ref{eqfrobfunc}) implies that $\sigma^{(r)}\not=0$. 
\end{proof}
This result allows us to prove the following.
\begin{lemma}\label{Lemmaprod} 
	Let $\tau$ be a smooth nonzero representation of $P_n$ and $\pi$ a smooth nonzero representation of $G_k$. Suppose all subrepresentations of $\tau$ have highest derivative at least $m$, then all subrepresentations of $\pi\times\tau$ have highest derivative at least $m$.
\end{lemma}
\begin{proof} 
 By assumption and the above proposition we have an injective map $$\tau\to\bigoplus_{j\geq m}(\hat{\Phi}^+)^{j-1}\Psi^+(\tau^{(j)}).$$ By exactness of $\pi\times\dotsc$ we see that 
 $\pi\times\tau\to\pi\times\bigoplus_{j\geq m}(\hat{\Phi}^+)^{j-1}\Psi^+(\tau^{(j)})$ is injective and note that $$\pi\times\bigoplus_{j\geq m}(\hat{\Phi}^+)^{j-1}\Psi^+(\tau^{(j)})\cong\bigoplus_{j\geq m}\pi\times(\hat{\Phi}^+)^{j-1}\Psi^+(\tau^{(j)}).$$
 As in 5.3 of \cite{Zelevinsky} we have an embedding $\pi\times(\hat{\Phi}^+)^{j-1}\Psi^+(\tau^{(j)})\hookrightarrow (\hat{\Phi}^+)^{j-1}\Psi^+((\pi\times\tau)^{(j)})$ for all $j\geq m$ and hence obtain an injective map $$\pi\times\tau\to\bigoplus_{j\geq m} (\hat{\Phi}^+)^{j-1}\Psi^+((\pi\times\tau)^{(j)}),$$
 which by the above proposition implies the result. 
\end{proof}

\begin{lemma}\label{zelvlemm} Let $\pi$ (and $\sigma$) be a smooth representation of $G_n$ (respectively $G_m$). 
\begin{enumerate}
    \item If $\omega$ is a nonzero submodule of $(\sigma\times\pi)|_{P_{m+n}},$ then there is $r\geq 0$ such that $(\Phi^{-})^r(\omega)\not=0$ and an embedding $(\Phi^{-})^{r}(\omega)\hookrightarrow\sigma^{(r)}\times\pi|_{P_n}$.
    \item Suppose that all nonzero subrepresentations of $\pi|_{P_n}$ have highest derivative at least $\mu$. Then all nonzero subrepresentations of $(\sigma\times\pi)|_{P_{m+n}}$ have highest derivative at least $\mu$. 
\end{enumerate}
\end{lemma}
\begin{proof}
ad (1):
Assume the contrary. By Proposition 4.13a) of ~\cite{BZI} there is a short exact sequence
$$0\to \sigma|_{P_m}\times\pi\to (\sigma\times\pi)|_{P_{n+m}}\to \sigma\times(\pi|_{P_n})\to 0,$$
and by our assumption $\omega$ cannot be contained in $\sigma^{(0)}\times(\pi|_{P_n})=\sigma\times(\pi|_{P_n})$. Hence we have an embedding $\omega\hookrightarrow\sigma|_{P_m}\times\pi$.
 For $i\geq 0$ suppose that $(\Phi^-)^{i}(\omega)\not=0$ and $(\Phi^-)^{i}(\omega)$ can be embedded into $(\Phi^-)^{i}(\sigma|_{P_m})\times\pi$. By Proposition 4.13 d) of ~\cite{BZI} we obtain that $(\Phi^-)^{i+1}(\omega)\not=0$. By Proposition 4.13c) of ~\cite{BZI} we have a short exact sequence
 $$0\to (\Phi^-)^{i+1}(\sigma|_{P_m})\times\pi\to\Phi^{-}( (\Phi^-)^{i}(\sigma|_{P_m})\times\pi)\to \sigma^{(i+1)}\times\pi|_{P_n}\to 0,$$
 and hence $(\Phi^-)^{i+1}(\omega)$ is contained in either $(\Phi^-)^{i+1}(\sigma|_{P_m})\times\pi$ or in $\sigma^{(i+1)}\times\pi|_{P_n}$. However, by our assumption the latter case is impossible and hence there is an embedding $(\Phi^-)^{i+1}(\omega)\hookrightarrow(\Phi^-)^{i+1}(\sigma|_{P_m})\times\pi$. 
 Inductively, we obtain that $(\Phi^-)^{m-1}(\omega)\not=0$ and an embedding $(\Phi^-)^{m-1}(\omega)\hookrightarrow(\Phi^-)^{m-1}(\sigma|_{P_m})\times\pi$. Again by Proposition 4.13 d) ~\cite{BZI} we see that $(\Phi^-)^{m}(\omega)\not=0$ and obtain an embedding
 $$(\Phi^-)^{m}(\omega)\hookrightarrow\Phi^{-}( (\Phi^-)^{m-1}(\sigma|_{P_m})\times\pi).$$
 However, $\Phi^{-}( (\Phi^-)^{m-1}(\sigma|_{P_m})\times\pi)$ is isomorphic to $\sigma^{(m)}\times\pi|_{P_n}$ which yields a contradiction.

ad (2): Let $\omega'$ be a nonzero subrepresentations of $(\sigma\times\pi)|_{P_{n+m}}$. By (1) there is an integer $r\geq 0$ such that such that $(\Phi^{-})^r(\omega')\not=0$ and there is an embedding $(\Phi^{-})^{r}(\omega')\hookrightarrow\sigma^{(r)}\times\pi|_{P_n}$. By Lemma \ref{Lemmaprod} every nonzero subrepresentation of $\sigma^{(r)}\times\pi|_{P_n}$ has highest derivative at least $\mu$. This implies that $\omega'$ has highest derivative at least $\mu+r$, which finishes the proof.
\end{proof}

\begin{lemma}\label{reduction}
	Let $\pi$ be a smooth finite length representation of $G_n$ and suppose that all nonzero subrepresentations of $\pi|_{P_n}$ have highest derivative at least $\mu$. Let $\Delta$ be the segment $[\rho,\nu\rho,\dotsc,\nu^{\alpha-1}\rho]$ and suppose that $\nu^{\alpha}\rho$ is not an element of the supercuspidal support of any irreducible subquotient of $\pi$.
 \begin{enumerate}
     \item For any permutation $s\in S_{\alpha}$ we set $\Sigma_s$ to be the representation $\Sigma_s=\nu^{s(1)-1}\rho\times\nu^{s(2)-1}\rho\times\dotsc\times \nu^{s(\alpha)-1}\rho$. Then all nonzero subrepresentations of $(\Sigma_s\times\pi)|_{P_{\alpha k+n}}$ have highest derivative at least $k+\mu$.
     \item Let $\sigma$ be an irreducible representation of $G_{\alpha k}$ with supercuspidal support $\Delta$. Then all nonzero subrepresentations of $(\sigma\times\pi)|_{P_{\alpha k+n}}$ have highest derivative at least $k+\mu$.
 \end{enumerate}
\end{lemma}
\begin{proof}
ad (1): Let $\omega$ be an irreducible subrepresentation of $(\Sigma_s\times\pi)|_{P_{\alpha k+n}}$ with highest derivative $l$. By Lemma \ref{zelvlemm} there is an integer $r\geq 0$ such that $(\Phi^-)^r(\omega)\not=0$ and $(\Phi^-)^r(\omega)\hookrightarrow\Sigma_s^{(r)}\times\pi|_{P_n}$.
	Firstly, suppose that $r=0$. Then $\omega^{(l)}$ is contained in
$(\Sigma_s\times(\pi|_{P_n}))^{(l)}$ which by Corollary 4.14 of \cite{BZI} equals $\Sigma_s\times\pi^{(l)}.$
	By Lemma 4.7 of \cite{BZI} this would imply that $\nu^\alpha\rho$ is an element of the supercuspidal support of an irreducible subquotient of $\Sigma_s\times\pi$, which contradicts our assumptions.
 
Note that for $1\leq j<k$ the representation $\Sigma_s^{(j)}$ is zero. Hence we obtain that $r\geq k$. By Lemma \ref{Lemmaprod} every nonzero irreducible subrepresentation of $\Sigma_s^{(r)}\times\pi|_{P_n}$ has highest derivative at least $\mu$ so in particular $(\Phi^-)^r(\omega)$ has highest derivative at least $\mu$. Then $\omega$ has highest derivative at least $k+\mu$.

ad (2): This follows immediately from part (1), since, by the theory of Bernstein and Zelevinsky, for any irreducible $G_{\alpha k}$-representation $\sigma$ with supercuspidal support $\Delta$ there exists a permutation $s\in S_\alpha$ such that there is an embedding of $\sigma$ into $\Sigma_s$.
\end{proof}

\begin{lemma}\label{higherder}
Let $\pi$ be a smooth representation of $G_n$ such that every nonzero subrepresentation of $\pi|_{P_n}$ has highest derivative at least $\mu$. Let $\Delta=[0,\alpha-1]_{\rho}$ be a segment of length at least two and suppose that $\nu^{\alpha}\rho$ is not an element of the supercuspidal support of any irreducible subquotient of $\pi$. Then all nonzero subrepresentations of $(Z([\nu^{\alpha-1}\rho],\Delta^-)\times\pi)|_P$ have highest derivative at least $2k+\mu$.
\end{lemma}
\begin{proof}
    Let $\omega$ be a nonzero irreducible subrepresentation of $(Z([\nu^{\alpha-1}\rho],\Delta^-)\times\pi)|_P$ with highest derivative $l\geq 1$. Lemma \ref{reduction} yields that $l\geq k+\mu$. By Lemma \ref{zelvlemm} there is an integer $r\geq 0$ such that $(\Phi^-)^r(\omega)\not=0$ and $$(\Phi^-)^r(\omega)\hookrightarrow Z([\nu^{\alpha-1}\rho],\Delta^-)^{(r)}\times\pi|_{P_n}.$$ 
    Note that for $r<2k$ we have that $Z([\nu^{\alpha-1}\rho],\Delta^-)^{(r)}=0$ unless $r$ is $0$ or $k$.
    
    If $r=0$ we see by Corollary 4.14 of \cite{BZI} that $\omega^{(l)}$ is contained in
	$$(Z([\nu^{\alpha-1}\rho],\Delta^-)\times\pi|_{P_n})^{(l)}=Z([\nu^{\alpha-1}\rho],\Delta^-)\times\pi^{(l)}.$$
	By Lemma 4.7 of \cite{BZI} this would imply that $\nu^\alpha\rho$ is an element of the supercuspidal support of an irreducible subquotient of $Z([\nu^{\alpha-1}\rho],\Delta^-)\times\pi$, which contradicts our assumptions.
 
    If $r=k$ we have an embedding $(\Phi^-)^k(\omega)\hookrightarrow Z([\nu^{\alpha-1}\rho],\Delta^-)^{(k)}\times\pi|_{P_n}$. By Lemma \ref{lemcomp} we have that
$$Z([\nu^{\alpha-1}\rho],\Delta^-)^{(k)}=Z([\nu^{\alpha-1}\rho],\Delta^{-,2})$$
and obtain an embedding 
 $$\omega^{(l)}\hookrightarrow Z([\nu^{\alpha-1}\rho],\Delta^{-,2})\times\pi^{(l-k)}.$$
As before, we would have by Lemma 4.7 of \cite{BZI} that $\nu^{\alpha}\rho$ is contained in the supercuspidal support of an irreducible subquotient of $Z([\nu^{\alpha-1}\rho],\Delta^-)\times\pi$, which again contradicts our assumptions.

Hence we obtain that $r\geq 2k$. By Lemma \ref{Lemmaprod} any nonzero subrepresentation of $$Z([\nu^{\alpha-1}\rho],\Delta^-)^{(r)}\times\pi|_{P_n}$$ has highest derivative at least $\mu$. In particular $(\Phi^-)^r(\omega)$ has highest derivative at least $\mu$, which immediately shows that $\omega$ has highest derivative at least $r+\mu\geq 2k+\mu$.
\end{proof}
\begin{lemma}\label{mirapart}
	Let $\sigma$ be an irreducible $P_n$-subrepresentation which can be embeded into a finite length $G_n$-representation $\pi$. Then the highest derivative sequence of $\sigma$ is a partition.
\end{lemma}
\begin{proof}
	We have that $\sigma$ is an irreducible representation of $\rho|_{P_n}$, where $\rho$ is a Jordan-H\"{o}lder factor of $\pi$. However, by Corollary 6.8 of \cite{Zelevinsky}, $\rho$ is a homogeneous representation with highest derivative $\rho^{(k)}$ for some natural number $k$. Then $\sigma^{(k)}\cong\rho^{(k)}$ and hence the classifying sequence of $\sigma$ coincides with the classifying partition of $\rho$.
\end{proof}

\subsection{Technical Result}
The following Proposition \ref{mainprop} is the main technical result we will need. Firstly we will set up some notation. For two cuspidal representations $\rho,\rho'$ of $G_k$ we say that $\rho'\geq\rho$ if either
\begin{itemize}
    \item there is no integer $\beta$ such that $\nu^\beta\rho\cong \rho'$,
    \item or if there is some integer $\beta$ such that $\nu^\beta\rho\cong\rho'$ then $\beta\geq 0$. 
    \end{itemize}
Let $\Delta'_1,\dotsc,\Delta'_r$ be a sequence of segments such that $\Delta'_i=[\rho_i,\nu\rho_i,\dotsc,\nu^{\alpha_i-1}\rho_i]$, where $\rho_i$ is a cuspidal representation of $G_{n_i}$. Set $n=\sum_{i=1}^r\alpha_in_i$. We assume that $\nu^{\alpha_i-1}\rho_i\geq\nu^{\alpha_j-1}\rho_j$ for $i\leq j$. In particular this implies that $\Delta'_i$ does not precede $\Delta'_j$ if $i<j$. Let $\lambda=(\lambda_1,\lambda_2,\dotsc)$ be the highest derivative partition of $Z(\Delta'_1,\dotsc,\Delta'_r)$. Note that $\lambda_i$ is the highest derivative of $Z((\Delta'_1)^{-,i-1},\dotsc,(\Delta'_r)^{-,i-1})$. 

For $1\leq i\leq r$ let $\pi_i$ be an irreducible representation such that the supercuspidal support of $\pi_i$ equals $\Delta'_i$.
We have that
$$\pi_i=Z(\Lambda^i_1,\dotsc,\Lambda^i_{k_i}),$$
where $\Lambda^i_j$ are disjoint segments whose union is $\Delta'_i$ and $\Lambda^i_{j+1}$ precedes $\Lambda^i_{j}$. We also set that $\pi_i^-=Z((\Lambda^i_1)^-,\dotsc,\Lambda^i_{k_i})$. Let
$$J(\pi_i)=J(\Lambda_1^i)\times J(\Lambda_2^i)\times\dotsc\times J(\Lambda_{k_i}^i)$$
and similarly
$$J(\pi_i^-)=J((\Lambda_1^i)^-)\times J(\Lambda_2^i)\times\dotsc\times J(\Lambda_{k_i}^i).$$
For any sequence of segments $\Delta_1',\dotsc,\Delta_r'$ and representations $\pi_1,\dotsc,\pi_r$ with conditions as above and if there is $1\leq s\leq r$ such that $\pi_s\not\cong Z(\Delta_s')$ we define the representation $\mathbb X(\pi_1,\dotsc,\pi_r,s)$ as
$$\left(\prod_{t=1}^{s-1}J(\pi_t)\times\left(Z([\nu^{\alpha_s-l}\rho_s],\Lambda_2^s)\times J({}^-\Lambda^s_1)\times \prod_{t=3}^{k_s}J(\Lambda_t^s)\right)\times \prod_{t=s+1}^{r}J(\pi_t)\right)|_{P}.$$

\begin{prop}\label{mainprop}
	Let $\omega$ be a nonzero irreducible subrepresentation of $(\pi_1\times\dotsc\times\pi_r)|_{P_n}$. If $\omega$ has highest derivative partition smaller than or equal to $\lambda$ (with respect to the lexicographic ordering), then $\pi_i=Z(\Delta'_i)$ for all $i=1,\dotsc,r$.
\end{prop}

\begin{proof}[Proof of Proposition \ref{mainprop}] Assume that not all $\pi_i$ are isomorphic to $Z(\Delta_i')$ for $1\leq i\leq r$. 
By the theory of Bernstein and Zelevinsky we have a map
$$\pi_i\hookrightarrow J(\pi_i).$$ 
Let $\Gamma$ be the set of integers $1\leq i\leq r$ where $k_i>1$, i.e.\ those integers $1\leq i\leq r$ such that $\pi_i\not\cong Z(\Delta'_i)$. By our assumption $\Gamma$ is nonempty. 
Choose $m\in\Gamma$ such that the length of $\Lambda_1^m$ is minimal among the lengths of $\Lambda_1^i$ for $i\in\Gamma$. Set $l$ to be the length of $\Lambda_1^m$, i.e.\ $\Lambda_1^m=[\nu^{\alpha_m-l}\rho_m,\dotsc,\nu^{\alpha_m-1}\rho_m]$. By our assumption, note that $l$ is strictly smaller than the length of $\Delta'_m$. The theory of Bernstein-Zelevinsky yields an embedding
$$\pi_m\hookrightarrow Z(\Lambda_1^m,\Lambda_2^m)\times J(\Lambda_3^m)\times\dotsc\times J(\Lambda^m_{k_m}).$$
Since ${}^-\Lambda^m_1=[\nu^{\alpha_m-l+1}\rho_m,\dotsc,\nu^{\alpha_m-1}\rho_m]$ and hence $\Lambda_1^m=[\nu^{\alpha_m-l}\rho_m]\cup{}^-\Lambda^m_1$, we obtain by repeatedly applying Lemma \ref{redpar} an embedding
$$Z(\Lambda_1^m,\Lambda_2^m)\hookrightarrow Z([\nu^{\alpha_m-l}\rho_m],\Lambda_2^m)\times J({}^-\Lambda^m_1).$$
Overall the above arguments yield an embedding of $\omega$ into $\mathbb X(\pi_1,\dotsc,\pi_r,m)$, i.e.\
\begin{equation}\label{eq2}
\left(\prod_{t=1}^{m-1}J(\pi_t)\times\left(Z([\nu^{\alpha_m-l}\rho_m],\Lambda_2^m)\times J({}^-\Lambda^m_1)\times \prod_{t=3}^{k_m}J(\Lambda_t^m)\right)\times \prod_{t=m+1}^{r}J(\pi_t)\right)|_{P_n}.
\end{equation}

However, by the following Lemma \ref{lem3} (whose proof appears at the end of this section), the highest derivative partition of $\omega$ has to be larger than $\lambda$ (with respect to the lexicographic ordering).
\end{proof}

\begin{lemma}\label{lem3}
Let $\sigma$ be an irreducible subrepresentation of $\mathbb X(\pi_1,\dotsc,\pi_r,m)$, i.e. Equation (\ref{eq2}), with highest derivative partition $\lambda^\sigma=(\lambda_1^\sigma,\lambda_2^\sigma,\dotsc)$. We have that $\lambda_i^\sigma\geq\lambda_i$ for all $1\leq i\leq l$ and either
 \begin{itemize}
     \item there is an $1\leq j<l$ such that $\lambda_j^\sigma>\lambda_j$,
     \item or $\lambda_l^\sigma>\lambda_l+n_m$.
 \end{itemize}

\end{lemma}
We will prove this result via induction on $l$, for which we need the following lemmata.
\begin{lemma}\label{lem1}
If $l=1$, i.e.\ $\Lambda_1^m=[\nu^{\alpha_m-1}\rho_m]$, any nonzero subrepresentation $\sigma$ of $\mathbb X(\pi_1,\dotsc,\pi_r,m)$ has highest derivative at least $\lambda_1+n_m$.
\end{lemma}

\begin{lemma}\label{lem2}
Suppose that $l>1$. Then any nonzero subrepresentation $\sigma$ of 
$\mathbb X(\pi_1,\dotsc,\pi_r,m)$
has highest derivative at least $\lambda_1$. If the highest derivative of $\sigma$ equals $\lambda_1$, we have that
$$\sigma^{(\lambda_1)}\hookrightarrow \mathbb X(\pi_1^-,\dotsc,\pi_r^-,m).$$
\end{lemma}

To prove Lemma \ref{lem1} and Lemma \ref{lem2} we introduce some new notation. For integers $1\leq\beta\leq r$ and $1\leq i\leq j$ if $j\leq k_\beta$ let
$Y^\beta_{i,j}$ be the representation $\prod_{t=i}^{j}J(\Lambda_{t}^\beta)$
and if $j>k_\beta$ we set $Y^\beta_{i,j}$ to be the trivial representation of $G_0=\{1\}$. Moreover, for integers $1\leq i\leq j\leq m$ we define
$$X_{i,j}\coloneqq\prod_{t=i}^{j}J(\pi_t)$$
if $m<i$ or $m>j$ and 
$$X_{i,j}\coloneqq\prod_{t=i}^{m-1}J(\pi_t)\times\left(Z([\nu^{\alpha_m-l}\rho_m],\Lambda_2^m)\times J({}^-\Lambda^m_1)\times
Y_{3,k_m}^m\right)\times\prod_{t=m+1}^{j}J(\pi_t)$$
if $i\leq m\leq j$.
If $l>1$ we set
$$  X_{i,j}^-\coloneqq
\prod_{t=i}^{j}J(\pi_t^-)$$
if $m<i$ or $m>j$ and 
$$X_{i,j}^-\coloneqq\prod_{t=i}^{m-1}J(\pi_t^-)\times\left(Z([\nu^{\alpha_m-l}\rho_m],\Lambda_2^m)\times J(({}^-\Lambda^m_1)^-)\times Y_{3,k_m}^m\right)\times\prod_{t=m+1}^{j}J(\pi_t^-)
$$
if $i\leq m\leq j$. Note that $X_{1,r}|_{P_n}=\mathbb X(\pi_1,\dotsc,\pi_r,m)$ and $X_{1,r}^-|_{P_n}=\mathbb X(\pi_1^-,\dotsc,\pi_r^-,m)$.

\begin{proof}[Proof of Lemma \ref{lem1}]
	
	Note that $J(\pi_r)^{(t)}=0$ for $1\leq t< n_r$ and hence any nontrivial $P_{\alpha_rn_r}$-submodule of $J(\pi_r)|_P$ has highest derivative at least $n_r$. By repeatedly applying Lemma \ref{reduction} we obtain that every nonzero $P_{\sum_{i=m+1}^r\alpha_in_i}$-submodule of $$X_{m+1,r}|_P=(J(\pi_{m+1})\times\dotsc\times J(\pi_r))|_P$$ has highest derivative at least $n_{m+1}+\dotsc+n_r$. By Lemma \ref{zelvlemm} 2) and Lemma \ref{higherder} any nonzero subrepresentation of
	$$X_{m,r}|_P=\left(Z([\nu^{\alpha_m-1}\rho_m],\Lambda_2^m)\times Y_{3,k_m}^m\times X_{m+1,r}\right)|_P$$ 
has to have highest derivative at least 
	$$2n_m+n_{m+1}+\dotsc+n_r.$$ From this we obtain by repeatedly applying Lemma \ref{reduction} that $\sigma$ has to have highest derivative at least 
	$$n_1+\dotsc+n_{m-1}+2n_m+n_{m+1}+\dotsc+n_r$$ which equals $\lambda_1+n_m$.
\end{proof}

\begin{proof}[Proof of Lemma \ref{lem2}]
Again note that $J(\pi_r)^{(t)}=0$ for $1\leq t< n_r$ and hence any nonzero $P_{\alpha_rn_r}$-submodule of $J(\pi_r)|_P$ has highest derivative at least $n_r$. A repeated application of Lemma \ref{zelvlemm} 2) and Lemma \ref{reduction} yields that $\sigma$ has highest derivative at least $\lambda_1$.\\
First we show that if $\sigma$ has highest derivative $\lambda_1$ then we have an embedding
 $$\sigma\hookrightarrow\left(X_{1,r}^-\times\prod_{t=1}^{r}\nu^{\alpha_t-1}\rho_t\right)|_{P_n}.$$
We will prove this using induction, so suppose that for some $1\leq i\leq r$ we have shown that there is an embedding
\begin{equation}\label{lem2eq1}
    \sigma\hookrightarrow\left(X_{1,i}\times X_{i+1,r}^-\times\prod_{t=i+1}^{r}\nu^{\alpha_t-1}\rho_t\right)|_{P_n}.
\end{equation}
We will then prove that we also have an embedding
$$\sigma\hookrightarrow\left(X_{1,i-1}\times X_{i,r}^-\times\prod_{t=i}^{r}\nu^{\alpha_t-1}\rho_t\right)|_{P_n}.$$
Firstly, for all $1\leq i\leq r$, we will show that\begin{equation}\label{lem2eq2}
    X_{i,i}\cong X_{i,i}^{-}\times\nu^{\alpha_i-1}\rho_i.
\end{equation}
Recall that for all $i\not=m$ we have
	$$X_{i,i}=J(\pi_i)=J(\Lambda_1^i)\times J(\Lambda_2^i)\times\dotsc\times J(\Lambda_{k_i}^i).$$ Let $l_i$ be the length of $\Lambda_1^i$ and hence $J(\Lambda_1^i)=\nu^{\alpha_i-l_i}\rho_i\times\nu^{\alpha_i-l_i+1}\rho_i\times\dotsc\times\nu^{\alpha_i-1}\rho_i.$ If $k_i=1$ we are done. If $k_i\geq 2$ note that for $2\leq j\leq k_i$ we have that $\nu^{\alpha_i-1}\rho_i$ is unlinked with all factors of $\Lambda_j^i$ (since, if $k_i\geq 2$, by assumption the length of $\Lambda_1^i$ is larger than one). This implies for all $2\leq j\leq k_i$ that
 $$\nu^{\alpha_i-1}\rho_i\times J(\Lambda_j^{i})\cong J(\Lambda_j^{i})\times\nu^{\alpha_i-1}\rho_i.$$
 and hence that
 $$X_{i,i}\cong X_{i,i}^{-}\times\nu^{\alpha_i-1}\rho_i.$$
If $i$ equals $m$, then
$$X_{m,m}=Z([\nu^{\alpha_m-l}\rho_m],\Lambda_2^m)\times J({}^-\Lambda^m_1)\times Y_{3,k_m}^m$$
and since $\nu^{\alpha_m-1}\rho_m$ is unlinked with all factors of $Y_{3,k_m}^m$ we analogously see that 
\begin{align*}
    J({}^-\Lambda^m_1)\times Y_{3,k_m}^m&\cong J(({}^-\Lambda^m_1)^-)\times\nu^{\alpha_m-1}\rho_m\times Y_{3,k_m}^m\\
    &\cong J(({}^-\Lambda^m_1)^-)\times Y_{3,k_m}^m\times\nu^{\alpha_m-1}\rho_m
\end{align*}
and hence
 $$X_{m,m}\cong X_{m,m}^{-}\times\nu^{\alpha_m-1}\rho_m.$$
By putting together Equations (\ref{lem2eq1}) and (\ref{lem2eq2}) we obtain an embedding 
$$\sigma\hookrightarrow\left(X_{1,i-1}\times X_{i,i}^{-}\times\nu^{\alpha_i-1}\rho_i\times X_{i+1,r}^-\times\prod_{t=i+1}^{r}\nu^{\alpha_t-1}\rho_t\right)|_{P_n}.$$
Suppose now that we have shown that there is an embedding 
$$\sigma\hookrightarrow\left(X_{1,i-1}\times X_{i,j}^-\times\nu^{\alpha_i-1}\rho_i\times X_{j+1,r}^-\times\prod_{t=i+1}^{r}\nu^{\alpha_t-1}\rho_t\right)|_{P_n}$$
for some $i\leq j\leq r-1$. We proceed via induction on $j$ and show that in this case there is also an embedding
$$\sigma\hookrightarrow\left(X_{1,i-1}\times X_{i,j+1}^-\times\nu^{\alpha_i-1}\rho_i\times X_{j+2,r}^-\times\prod_{t=i+1}^{r}\nu^{\alpha_t-1}\rho_t\right)|_{P_n}.$$
There are four cases.
\begin{itemize}
\item If $\nu^{\alpha_i-1}\rho_i\not\cong\nu^{\alpha_{j+1}-1}\rho_{j+1}$, then $\nu^{\alpha_i-1}\rho_i$ is unlinked with all factors in $X_{j+1,j+1}^-$ and we obtain that
$$\nu^{\alpha_i-1}\rho_i\times X_{j+1,j+1}^-\cong X_{j+1,j+1}^-\times\nu^{\alpha_i-1}\rho_i$$
and hence
$$\sigma\hookrightarrow\left(X_{1,i-1}\times X_{i,j+1}^-\times\nu^{\alpha_i-1}\rho_i\times X_{j+2,r}^-\times\prod_{t=i+1}^{r}\nu^{\alpha_t-1}\rho_t\right)|_{P_n}.$$

\item If $\nu^{\alpha_i-1}\rho_i\cong\nu^{\alpha_{j+1}-1}\rho_{j+1}$ and $j+1\not=m$ we have that $X_{j+1,j+1}^-=J(\pi_{j+1}^-)$ and obtain
$$\nu^{\alpha_i-1}\rho_i\times X_{j+1,j+1}^-\cong J((\Lambda^{j+1}_1)^{-,2} )\times\nu^{\alpha_i-1}\rho_i\times\nu^{\alpha_{j+1}-2}\rho_{j+1}\times Y^{j+1}_{2,k_{j+1}}.$$
Since $\nu^{\alpha_i-1}\rho_i\cong\nu^{\alpha_{j+1}-1}\rho_{j+1}$ note that $\nu^{\alpha_i-1}\rho_i\times\nu^{\alpha_{j+1}-2}\rho_{j+1}$ is reducible with the two Jordan-H\"{o}lder components $Z([\nu^{\alpha_i-1}\rho_i],[\nu^{\alpha_i-2}\rho_i])$ and $Z([\nu^{\alpha_i-2}\rho_i,\nu^{\alpha_i-1}\rho_i])$.
Write $\Theta$ for
$$Y^{j+1}_{2,k_{j+1}}\times X_{j+2,r}^-\times\prod_{t=i+1}^{r}\nu^{\alpha_t-1}\rho_t.$$
Then there is an embedding of $\sigma$ into 

$$\left(X_{1,i-1}\times X_{i,j}^-\times J((\Lambda^{j+1}_1)^{-,2} )\times Z([\nu^{\alpha_i-1}\rho_i],[\nu^{\alpha_i-2}\rho_i])\times\Theta\right)|_{P_n}$$
or into
$$\left(X_{1,i-1}\times X_{i,j}^-\times J((\Lambda^{j+1}_1)^{-,2})\times Z([\nu^{\alpha_i-2}\rho_i,\nu^{\alpha_i-1}\rho_i])\times\Theta\right)|_{P_n}.$$
By Lemma \ref{reduction} any nonzero subrepresentation of $(\prod_{t=i+1}^{r}\nu^{\alpha_t-1}\rho_t)|_P$ has highest derivative at least $n_{i+1}+\dotsc+n_r$ and then by Lemma \ref{zelvlemm} 2) any nonzero subrepresentation of $\Theta$ also has highest derivative at least $n_{i+1}+\dotsc+n_r$. By Lemma \ref{higherder} any nonzero subrepresentation of 
$$(Z([\nu^{\alpha_i-1}\rho_i],[\nu^{\alpha_i-2}\rho_i])\times\Theta)|_{P}$$
has highest derivative at least $2n_i+n_{i+1}+\dotsc+n_r$. By applying Lemma \ref{zelvlemm} 2) and Lemma \ref{reduction} again we obtain that any nonzero subrepresentation of
 $$\left(X_{1,i-1}\times X_{i,j}^-\times J((\Lambda^{j+1}_1)^{-,2} )\times Z([\nu^{\alpha_i-1}\rho_i],[\nu^{\alpha_i-2}\rho_i])\times\Theta\right)|_{P_n}$$
 has highest derivative at least $\lambda_1+n_i$. However, since by assumption $\sigma$ has highest derivative $\lambda_1$ we see that $\sigma$ can be embedded into
$$\left(X_{1,i-1}\times X_{i,j}^-\times J((\Lambda^{j+1}_1)^{-,2})\times Z([\nu^{\alpha_i-2}\rho_i,\nu^{\alpha_i-1}\rho_i])\times\Theta\right)|_{P_n}.$$
Since 
$$Z([\nu^{\alpha_i-2}\rho_i,\nu^{\alpha_i-1}\rho_i])\hookrightarrow \nu^{\alpha_i-2}\rho_i\times\nu^{\alpha_i-1}\rho_i$$
and 
$$\nu^{\alpha_i-1}\rho_i\times\prod_{t=2}^{k_{j+1}}J(\Lambda^{j+1}_t)\cong\prod_{t=2}^{k_{j+1}}J(\Lambda^{j+1}_t)\times\nu^{\alpha_i-1}\rho_i$$
we obtain an embedding of $\sigma$ into
$$\left(X_{1,i-1}\times X_{i,j+1}^-\times\nu^{\alpha_i-1}\rho_i\times X_{j+2,r}^-\times\prod_{t=i+1}^{r}\nu^{\alpha_t-1}\rho_t\right)|_{P_n}.$$
\item If $j+1=m,l>2$ and $\nu^{\alpha_i-1}\rho_i\cong\nu^{\alpha_{m}-1}\rho_{m}$ recall that 
$$X_{m,m}^-=Z([\nu^{\alpha_m-l}\rho_m],\Lambda_2^m)\times J(({}^-\Lambda^m_1)^-)\times Y_{3,k_m}^m.$$ Then $\nu^{\alpha_i-1}\rho_i$ is unlinked with $[\nu^{\alpha_m-l}\rho_m]$ and $\Lambda_2^m$ and all factors of $J(({}^-\Lambda^m_1)^{-,2})$. Write 
$\Xi_1$ for 
$$Z([\nu^{\alpha_m-l}\rho_m],\Lambda_2^m)\times J(({}^-\Lambda^m_1)^{-,2})$$
and $\Xi_2$ for
$$Y_{3,k_m}^m\times X_{j+2,r}^-\times\prod_{t=i+1}^{r}\nu^{\alpha_t-1}\rho_t.$$
We obtain an embedding of $\sigma$ into
$$\left(X_{1,i-1}\times X_{i,j}^-\times \Xi_1\times \nu^{\alpha_i-1}\rho_i\times\nu^{\alpha_{i}-2}\rho_{i}\times \Xi_2\right)|_{P_n}.$$
By Lemma \ref{zelvlemm} 2) and Lemma \ref{reduction} any nonzero subrepresentation of $\Xi_2|_P$ has highest derivative at least $n_{i+1}+\dotsc+n_r$. Lemma \ref{higherder} implies that any nonzero subrepresentation of $$\left(Z([\nu^{\alpha_i-1}\rho_i],[\nu^{\alpha_{i}-2}\rho_{i}])\times \Xi_2\right)|_P$$ has highest derivative at least $2n_i+n_{i+1}+\dotsc+n_r$
and by Lemma \ref{zelvlemm} 2) and Lemma \ref{reduction} we then obtain that any nonzero subrepresentation of 
$$\left(X_{1,i-1}\times X_{i,j}^-\times \Xi_1\times Z([\nu^{\alpha_i-1}\rho_i],[\nu^{\alpha_{i}-2}\rho_{i}])\times \Xi_2\right)|_{P_n}.$$
has highest derivative at least $\lambda_1+n_i$. By assumption we then have an embedding of $\sigma$ into 
$$\left(X_{1,i-1}\times X_{i,j}^-\times\Xi_1\times Z([\nu^{\alpha_{i}-2}\rho_{i},\nu^{\alpha_i-1}\rho_i]) \times \Xi_2\right)|_{P_n}.$$
Since $Z([\nu^{\alpha_i-2}\rho_i,\nu^{\alpha_i-1}\rho_i])\hookrightarrow \nu^{\alpha_i-2}\rho_i\times\nu^{\alpha_i-1}\rho_i$ and $\nu^{\alpha_i-1}\rho_i\times Y_{3,k_m}^m\cong Y_{3,k_m}^m\times\nu^{\alpha_i-1}\rho_i$ we obtain an embedding 
$$\sigma\hookrightarrow\left(X_{1,i-1}\times X_{i,m}^-\times\nu^{\alpha_i-1}\rho_i\times X_{m+1,r}^-\times\prod_{t=i+1}^{r}\nu^{\alpha_t-1}\rho_t\right)|_{P_n}$$
\item If $j+1=m,l=2$ and $\nu^{\alpha_i-1}\rho_i\cong\nu^{\alpha_{m}-1}\rho_{m}$ then 
$$X_{m,m}^-=Z([\nu^{\alpha_m-2}\rho_m],\Lambda_2^m)\times Y_{3,k_m}^m$$ and the representation
	$$\nu^{\alpha_{i}-1}\rho_{i}\times Z([\nu^{\alpha_m-2}\rho_m],\Lambda^m_2)$$
 is reducible and by Lemma \ref{weirdcase} has the two Jordan-H\"{o}lder components
	$$Z([\nu^{\alpha_i-1}\rho_i],[\nu^{\alpha_i-2}\rho_i],\Lambda^m_2)$$ and 
	$$Z([\nu^{\alpha_i-1}\rho_i,\nu^{\alpha_i-2}\rho_i],\Lambda^m_2).$$
 Write $\Upsilon$ for the representation $$Y_{3,k_m}^m\times X_{m+1,r}^-\times\prod_{t=m}^{r}\nu^{\alpha_t-1}\rho_t$$ and note that by Lemma \ref{zelvlemm} 2) and Lemma \ref{reduction} any nonzero subrepresentation of $\Upsilon|_P$ has highest derivative at least $n_m+\dotsc+n_{r}$. Then $\sigma$ is contained in either
 \begin{equation*}
     \left(X_{1,i-1}\times X_{i,m-1}^-\times Z([\nu^{\alpha_i-1}\rho_i],[\nu^{\alpha_i-2}\rho_i],\Lambda^m_2)\times\Upsilon\right)|_{P_n}
 \end{equation*}
 or in
 \begin{equation}\label{fourcastrue}
    \left(X_{1,i-1}\times X_{i,m-1}^-\times Z([\nu^{\alpha_i-1}\rho_i,\nu^{\alpha_i-2}\rho_i],\Lambda^m_2)\times\Upsilon\right)|_{P_n}.
 \end{equation}
In the first case note that by the theory of Bernstein and Zelevinsky that $$Z([\nu^{\alpha_i-1}\rho_i],[\nu^{\alpha_i-2}\rho_i],\Lambda^m_2)\hookrightarrow Z([\nu^{\alpha_i-1}\rho_i],[\nu^{\alpha_i-2}\rho_i])\times Z(\Lambda^m_2)$$
and hence we have an embedding of $\sigma$ into
\begin{equation}\label{fourcas2}
    \left(X_{1,i-1}\times X_{i,m-1}^-\times Z([\nu^{\alpha_i-1}\rho_i],[\nu^{\alpha_i-2}\rho_i])\times Z(\Lambda^m_2)\times\Upsilon\right)|_{P_n}.
\end{equation}
However, by Lemma \ref{zelvlemm} 2) and Lemma \ref{higherder} any nonzero subrepresentation of
$$\left(Z([\nu^{\alpha_i-1}\rho_i],[\nu^{\alpha_i-2}\rho_i])\times Z(\Lambda^m_2)\times\Upsilon\right)|_P$$
has highest derivative at least $2n_{i}+n_{i+1}+\dotsc+n_{r}$. By Lemma \ref{zelvlemm} and Lemma \ref{reduction} any nonzero subrepresentation of (\ref{fourcas2}) has highest derivative at least $\lambda_1+n_{i}$. Hence $\sigma$ is contained in (\ref{fourcastrue}), which by Lemma \ref{redpar} implies that $\sigma$ can be embedded into 
$$\left(X_{1,i-1}\times X_{i,m-1}^-\times Z([\nu^{\alpha_i-2}\rho_i],\Lambda^m_2)\times \nu^{\alpha_i-1}\rho_i\times\Upsilon\right)|_{P_n}.$$
Since $\nu^{\alpha_i-1}\rho_i$ and all factors of $Y_{3,k_m}^m$ are unlinked we obtain an embedding
$$\sigma\hookrightarrow\left(X_{1,i-1}\times X_{i,m}^-\times\nu^{\alpha_i-1}\rho_i\times X_{m+1,r}^-\times\prod_{t=i+1}^{r}\nu^{\alpha_t-1}\rho_t\right)|_{P_n}.$$
\end{itemize}
Induction on $j$ shows that 
$$\sigma\hookrightarrow\left(X_{1,i-1}\times X_{i,r}^-\times\prod_{t=i}^{r}\nu^{\alpha_t-1}\rho_t\right)|_{P_n}$$
and via induction on $i$ we obtain that
$$\sigma\hookrightarrow\left(X_{1,r}^-\times\prod_{t=1}^{r}\nu^{\alpha_t-1}\rho_t\right)|_{P_n}.$$
By Lemma \ref{zelvlemm} 1) there is a $\gamma\geq 0$ such that $(\Phi^-)^{\gamma}(\sigma)\not=0$ and $(\Phi^-)^{\gamma}(\sigma)$ can be embedded into
$$\left(X_{1,r}^-\right)^{(\gamma)}\times\left(\left[\prod_{t=1}^{r}\nu^{\alpha_t-1}\rho_t\right]|_P\right).$$
 By Lemma \ref{reduction} any nonzero subrepresentation of $(\prod_{t=1}^{r}\nu^{\alpha_t-1}\rho_t)|_P$ has highest derivative at least $\lambda_1$, which implies that $\sigma$ has highest derivative at least $\lambda_1+\gamma$. Since by assumption $\sigma$ has highest derivative $\lambda_1$ we see that $\gamma$ equals zero and hence
$$\sigma\hookrightarrow X_{1,r}^-\times\left(\left[\prod_{t=1}^{r}\nu^{\alpha_t-1}\rho_t\right]|_P\right).$$
 By taking the $\lambda_1$-derivative and Corollary 4.14 of ~\cite{BZI} we obtain that 
$$\sigma^{(\lambda_1)}\hookrightarrow X_{1,r}^-$$
which finishes the proof.
\end{proof}

\begin{proof}[Proof of Lemma \ref{lem3}]
	We proceed via induction on $l$, the length of $\Lambda_1^m$. If $l=1$ this follows from Lemma \ref{lem1}. Suppose now that $l>1$. By Lemma \ref{lem2} we have that $\lambda_1^\sigma\geq\lambda_1$. If $\lambda_1^\sigma>\lambda_1$ we are done. If $\lambda_1^\sigma=\lambda_1$ we can apply the second part of Lemma \ref{lem2} and obtain an embedding
	$$\sigma^{(\lambda_1)}\hookrightarrow\mathbb X(\pi_1^-,\dotsc,\pi_r^-,m).$$
	Let $\sigma'$ be an irreducible subrepresentation of $\sigma^{(\lambda_1)}|_P$ and $\lambda^{\sigma'}=(\lambda_1^{\sigma'},\lambda_2^{\sigma'},\dotsc)$ be the highest derivative partition of $\sigma'$. Clearly we have an embedding $\sigma'\hookrightarrow\mathbb X(\pi_1^-,\dotsc,\pi_r^-,m)$. Since $\pi_m^-=Z((\Lambda_1^m)^-,\Lambda_2^m,\dotsc,\Lambda_{k_m}^m)$ and the length of $(\Lambda_1^m)^-$ is $l-1$ we can apply our induction hypothesis. We obtain that $\lambda_i^{\sigma'}\geq\lambda_{i+1}$ for $1\leq i\leq l-1$ and either:
	\begin{itemize}
		\item there exits $1\leq j<l-1$ such that $\lambda_j^{\sigma'}>\lambda_{j+1}$,
		\item or $\lambda_{l-1}^{\sigma'}>\lambda_l+n_m$.
	\end{itemize}  
	In the first case 
	$$(((\sigma')^{(\lambda_2)})^{(\lambda_3)\dotsc})^{(\lambda_j^{\sigma'})}\not=0$$ 
	and hence  $$(((\sigma')^{(\lambda_2)})^{(\lambda_3)\dotsc})^{(\lambda_j^{\sigma'})}\hookrightarrow((\sigma^{(\lambda_1)})^{(\lambda_2)\dotsc})^{(\lambda_j^{\sigma'})}\not=0.$$
	If $\lambda_j^{\sigma'}=\lambda_{j+1}$ for $1\leq j<l-1$ we have 
	$$(((\sigma')^{(\lambda_2)})^{(\lambda_3)\dotsc})^{(\lambda_l+n_m)}\not=0$$ 
	and hence
	$$(((\sigma')^{(\lambda_2)})^{(\lambda_3)\dotsc})^{(\lambda_l+n_m)}\hookrightarrow((\sigma^{(\lambda_1)})^{(\lambda_2)\dotsc})^{(\lambda_l+n_m)}\not=0.$$
    In both cases we see that $(\lambda_1,\lambda_2,\dotsc)$ cannot be the highest derivative partition of $\sigma$ which implies the result.
\end{proof}

We are now almost in a position to prove the Theorem \ref{thm:truncate}. The only point missing is the following lemma.

\subsection{Proof of Theorem \ref{thm:truncate}}
\begin{lemma}
Let $\pi$ be a finitely generated representation of $G_n$. Suppose that each summand in the cosocle of $\pi$ has highest derivative partition not smaller than or equal to $\lambda$. Then $\pi^{\preceq\lambda}=0$.
\end{lemma}
\begin{proof}
Let $\pi'$ be the image of $\bigoplus_{\lambda'\not\preceq\lambda}W'_\lambda\otimes\Hom(W'_\lambda,\pi)\to\pi$, i.e.\ $\pi^{\preceq\lambda}=\pi/\pi'$. If $\pi^{\preceq\lambda}$ is not zero it has an irreducible quotient $\sigma$, i.e.\ we have a surjective map $\pi^{\preceq\lambda}\twoheadrightarrow\sigma$. However, then clearly $\sigma$ also appears as a quotient of $\pi$ and hence by assumption has highest derivative partition $\lambda_\sigma$ not smaller than or equal to $\lambda$. We have a nonzero map $W'_{\lambda_\sigma}\to\sigma$ which by the projectivity of $W'_{\lambda_\sigma}$ gives rise to a nonzero map $W'_{\lambda_\sigma}\to\pi$, whose image is not contained in $\pi'$. This yields a contradiction.
\end{proof}
\begin{proof}[Proof of Theorem \ref{thm:truncate}]
The statement we want to prove is that for $1\leq t\leq r$ let $\Delta_t$ be segments such that $\Delta_i$ does not precede $\Delta_j$ for any $i>j$. Moreover, if for some $1\leq i,j\leq r$ we have that $\Delta_i\supset\Delta_j$ we assume that $i>j$. Let $\lambda$ be the highest derivative partition of $Z(\Delta_1,\dotsc,\Delta_r)$. Then the surjection
	$$I(\Delta_1)\times\dotsc\times I(\Delta_r)\twoheadrightarrow Z(\Delta_1)\times\dotsc\times Z(\Delta_r)$$
	induces an isomorphism 
	$$\left(I(\Delta_1)\times\dotsc\times I(\Delta_r)\right)^{\preceq\lambda}\cong  Z(\Delta_1)\times\dotsc\times Z(\Delta_r).$$
Let $\Omega$ be the kernel of the map
$$I(\Delta_1)\times\dotsc\times I(\Delta_r)\twoheadrightarrow Z(\Delta_1)\times\dotsc\times Z(\Delta_r).$$
Then, since the functor $V\mapsto V^{\preceq\lambda}$ is right exact, it is enough to prove that $\Omega^{\preceq\lambda}=0$.
    By the above Lemma this would follow if any summand of the cosocle of $\Omega$ has highest derivative partition not smaller than or equal to $\lambda$. By taking contragredients it is hence enough to show that if $\omega$ is an irreducible subrepresentation of $\widetilde{\Omega}$, then $\omega$ has highest derivative partition not smaller than or equal to $\lambda$. Now any irreducible subrepresentation of $\widetilde{\Omega}$ is contained in a representation of the form $\pi_1\times\dotsc\times\pi_r$ where $\pi_i$ are irreducible representations with supercuspidal supports $\widetilde{\Delta_i}$ and who are not all isomorphic to $Z(\widetilde{\Delta_i})$. The result then follows from Proposition \ref{mainprop}.
\end{proof}


\begin{thebibliography}{}
\bibitem[BR]{rumelhart}
J. Bernstein, {\em Representations of $p$-adic groups,} Harvard University, 1992.  Lectures by Joseph Bernstein; written by Karl Rumelhart.

\bibitem[BZ]{BZI}
J. Bernstein and A. Zelevinsky, {\em Induced representations of reductive $p$-adic groups. I,}
Ann. Sci. {\'E}c. Norm. Sup. {\bf 4} 10 (1977), no. 4, 441--472.

\bibitem[Ch]{chan}
K. Y. Chan, {\em On the product functor on inner forms of the general linear group over a non-archimedean field,} Transformation Groups (2024).

\bibitem[Gi]{doubling}
J. Girsch, {\em The twisted doubling method in algebraic families}, preprint, arxiv:2410.22525.

\bibitem[GGS]{GGS}
R. Gomez, D. Gourevitch, and S. Sahi, {\em Generalized and degenerate Whittaker models,}
Comp. Math. {\bf 153} (2017), no. 2, 223--256.

\bibitem[Go]{gourevich appendix}
D. Gourevich, {\em Generalized Whittaker and Zelevinsky models for the general linear group,}
appendix to E. Lapid and O. Offen, {\em Explicit Plancherel formula for the space of symplectic forms,} IMRN {\bf 18} (2022), 14255--14294.

\bibitem[GS]{GS}
D. Gourevitch and S. Sahi, {\em Annihilator varieties, adduced representations, Whittaker functionals, and rank for unitary representations of $\GL(n)$,}
Selecta Math. {\bf 19} (2013), no. 1, 141--172.

\bibitem[He]{bernstein}
Helm, D. {\em The Bernstein center of the category of smooth $W(k)[\GL_n(F)]$-modules,}
Forum Math. Sigma {\bf 4} (2016), e11.

\bibitem[He2]{whittaker}
Helm, D. {\em Whittaker modules and the integral Bernstein center for $\GL_n$,}
Duke Math. J. {\bf 165} (2016), no. 9, 1597--1628.

\bibitem[MS]{MS}
A. M{\'i}nguez and V. S{\'e}cherre, {\em Repr{\'e}sentations lisses modulo $\ell$ de $\GL_m(D),$}
Duke Math. J. {\bf 163} (2014), no. 4, 795--887.

\bibitem[MW]{MW}
C. Moeglin and J. L. Waldspurger, {\em Modeles de Whittaker degeneres pour des groupes $p$-adiques,}
Math. Z. {\bf 196} (1987), no. 3, 427--452.

\bibitem[SZ]{schneider-zink}
P. Schneider and T. Zink, {$K$-types for the tempered components of a $p$-adic general linear group,} J. Reine Angew. Math. {\bf 517} (1999), 161--208.

\bibitem[Va]{varma}
S. Varma, {\em On a result of Moeglin and Waldspurger in residual characteristic two,} Math. Z. {\bf 277} (2014), no. 3-4, 1027--1048.

\bibitem[Ze]{Zelevinsky}
A. Zelevinsky, {\em Induced representations of reductive $p$-adic groups. II.  On irreducible representations of $\GL_n$,} Ann. Sci. {\'E}c. Norm. Sup. {\bf 4} 13 (1980), no. 2, 165--210.

\end{thebibliography}
\end{document}